\documentclass[]{amsart}
\usepackage{amssymb}
\usepackage{yfonts}
\usepackage{graphicx}
\usepackage{txfonts}
\usepackage{calrsfs}
\usepackage{hyperref}
\newtheorem{theorem}{Theorem}
\newtheorem{corollary}{Corollary}

\newtheorem{example}{Example}
\newtheorem{lemma}{Lemma}

\newtheorem{remark}{Remark}
\newcommand{\eps}{\varepsilon}
\DeclareMathAlphabet{\mathpzc}{OT1}{pzc}{m}{it}

\DeclareMathOperator{\essinf}{ess\;inf}
\DeclareMathOperator{\Graph}{Gr}
\DeclareMathOperator{\fix}{Fix}
\DeclareMathOperator{\F}{F}

\DeclareMathOperator{\h}{H}

\DeclareMathOperator{\co}{co}

\DeclareMathOperator{\modulus}{mod}

\DeclareMathOperator{\ev}{ev}

\newcommand{\w}{\tilde}
\newcommand{\map}{\multimap}
\newcommand{\<}{\leqslant}
\newcommand{\mf}{{\Phi}}
\newcommand{\n}{{n\geqslant 1}}

\newcommand{\R}[1]{\varmathbb{R}^{#1}}
\newcommand{\Hi}{\mathbb{H}}
\newcommand{\f}{\left}
\newcommand{\g}{\right}

\newcommand{\res}[2]{#1\:\rule[-1.5mm]{0.45pt}{4mm}\,\rule[-1mm]{0mm}{4mm}_{#2}}

\begin{document}
\title[Acyclicity of the solution set of two-point boundary value problems]{Acyclicity of the solution set of two-point boundary value problems for second order multivalued differential equations}
\author{Rados\l aw Pietkun}
\address{Toru\'n, Poland}
\email{rpietkun@pm.me}

\begin{abstract}
The topological and geometrical structure of the set of solutions of two-point boundary value problems for second order differential inclusions in Banach spaces is investigated. It is shown that under the Carath\'eodory-type assumptions the solution set of the periodic boundary value problem is nonempty compact acyclic in the space of continuously differentiable functions as well as in the Bochner-Sobolev space $\Hi^2$ endowed with the weak topology. The proof relies heavily on the accretivity of the right-hand side of differential inclusion. The Lipschitz case is treated separately. As one might expect the solution set is, in this case, an absolute retract.
\end{abstract}
\maketitle


\section{Introduction}
It is known that the Cauchy problem with continuous right-hand side possesses local solutions although the uniqueness property does not hold in general. This observation made by Peano became the starting point for investigating the topology of solutions of initial value problems. A precise topological characterization of the solution set was found in 1942 by N. Aronszajn, who improved the results of Kneser and Hukuhara by showing that the Peano funnel is an $R_\delta$-set. This means notably that in the absence of lipschitzianity of the right-hand side of the respective differential equation, the set of all solutions may not be a singleton but, from the point of view of algebraic topology, it is equivalent to a one point space. In 1986 De Blasi and Myjak generalized Aronszajn's theorem to the case of differential inclusions with usc convex valued right-hand sides. Since then showed up an overwhelming number of papers devoted to the study of the structure of the solution set for differential equations and inclusions. After rich and extensive bibliography on the subject we refer the reader to the monograph \cite{djebali}.\par Nevertheless, the matter of description of the topology of solutions to other boundary value problems had been taken so far relatively rare. Some insight of what has been achieved in this area gives the overview contained in \cite[III.3.]{andres}. It is worth noting that there are no reliable results concerning topological properties of the solution set of so-called nonlocal Cauchy problems for differential equations with the right side, which is not Lipschitz continuous.\par The purpose of this note is to prove results describing the topological and geometrical properties of the set $S_{\!F}$ of all solutions of two-point boundary value problems for second order differential inclusions defined in an abstract Banach space $E$. In Section 3. we show that the solution set of the following periodic boundary value problem
\begin{equation}\label{inclusion}
\begin{gathered}
x''+a_1(t)x'+a_0(t)x\in F(t,x),\;\text{a.e. on } I:=[0,1],\\
x(0)=x(1),\\
x'(0)=x'(1),
\end{gathered}
\end{equation}
where $a_1,a_0\colon I\to\R{}$ and $F\colon I\times E\map E$, is nonempty compact acyclic as a subset of $C^1(I,E)$ as well as a subset of the space of solutions $\Hi^2(0,1;E)$ provided this space is equipped with the weak topology and the right-hand side $F$ is a convex valued weak upper Carath\'eodory multimap. Our approach consists in replacing the problem (1) with an equivalent integral problem by the use of Green's function for reduced system and applying the Browder-Gupta type result characterizing the set of fixed points of an appriopriate nonlinear operator. A key role in this approach plays an accretivity assumption regarding the right-hand side $F$, which entails uniqueness of solutions to problems constituting the approximation of the original problem.\par As one knows, the solution set of the Cauchy problem associated to a differential inclusion $\dot{x}\in F(t,x)$ is contractible provided $F$ admits a measurable-locally Lipschitz selection. If the multivalued right-hand side $F$ is simply Lipschitz continuous and possesses not necessarily convex values, then with the aid of theorem \cite[Th.1.]{bres} it can be shown that the set of derivatives of all Carath\'eodory solutions of the Cauchy problem is a retract of the Bochner space $L^1(I,E)$. The set of solutions for the initial value problem is nothing but a continuous image through the integral operator. Therefore, this set is always at least arcwise connected. In Section 4. our goal is to present a detailed proof of the fact that the solution set of the following two-point boundary value problem
\begin{equation}\label{inclus}
\begin{gathered}
a_2(t)x''+a_1(t)x'+a_0(t)x\in F(t,x)\;\;\;\text{a.e. on } I,\\
b_{11}x(0)+b_{12}x'(0)+c_{11}x(1)+c_{12}x'(1)=d_1,\\
b_{21}x(0)+b_{22}x'(0)+c_{21}x(1)+c_{22}x'(1)=d_2,
\end{gathered}
\end{equation}
where $a_2,a_1,a_0\colon I\to\R{}$, $a_2(t)\neq 0$ on $I$ and $F$ is a measurable Lipschitz multivalued map, is a retract of the Bochner-Sobolev space $W^{2,1}(0,1;E)$. Our reasoning is also based on the result concerning the set of fixed points of a multivalued contraction with decomposable values.  
\section{Preliminaries}
Let $(E,|\cdot|)$ be a Banach space, $E^*$ its normed dual and $\sigma(E,E^*)$ its weak topology. Then $J\colon E\to 2^{E^*}$, defined by \[J(x):=\f\{x^*\in E^*\colon |x^*|=|x|\;\;\text{and}\;\;x^*(x)=|x|^2\g\}\] is called the duality map of $E$. The semi-inner products $\langle\cdot,\cdot\rangle_\pm\colon E\times E\to\R{}$ are given by the formulas \[\langle x,y\rangle_+:=\max\{y^*(x)\colon y^*\in J(y)\}\;\;\text{and}\;\;\langle x,y\rangle_-:=\min\{y^*(x)\colon y^*\in J(y)\}\] (for more information about these notions consult \cite{barbu,deimling,deimling2,papa}). The (normed) space of bounded linear operators from a Banach space $E_1$ to a Banach space $E_2$ is denoted by $\mathcal{L}(E_1,E_2)$. Given $T\in\mathcal{L}(E_1,E_2)$, $||T||_{{\mathcal L}}$ is the norm of $T$. For any $\eps>0$ and $A\subset E$, $B(A,\eps)$ ($D(A,\eps)$) is an open (closed) $\eps$-neighbourhood of the set $A$. The closure and the closed convex envelope of $A$ will be denoted by $\overline{A}$ and $\overline{\co} A$, respectively. If $x\in E$ we set $d(x,A):=\inf\{|x-y|\colon y\in A\}$. Besides, for two nonempty closed bounded subsets $A$, $B$ of $E$ we denote by $d_E(A,B)$ the Hausdorff distance from $A$ to $B$, i.e. $d_E(A,B):=\max\{\sup\{d(x,B)\colon x\in A\},\sup\{d(y,A)\colon y\in B\}\}$.\par We denote by $(C(I,E),||\cdot||)$ (resp. $(C^1(I,E),||\cdot||_{C^1})$) the Banach space of all continuous (resp. continuously differentiable) maps $I\to E$ equipped with the maximum norm (resp. $||x||_{C^1}=||x||+||\dot{x}||$). Let $1\<p<\infty$. By $(L^p([a,b],E),||\cdot||_p)$ we mean the Banach space of all (Bochner) $p$-integrable maps $f\colon[a,b]\to E$, i.e. $f\in L^p([a,b],E)$ iff map f is strongly measurable and \[||f||_p=\left(\int_a^b|f(t)|^p\,dt\right)^{\frac{1}{p}}<\infty.\] Recall that strong measurability is equivalent to the usual measurability in case $E$ is separable. A subset $K\subset L^1(I,E)$ is called decomposable if for every $u, w\in K$ and every Lebesgue measurable $A\subset I$ we have $u\cdot\chi_A+w\cdot\chi_{I\setminus A}\in K$. Recall that the Bochner-Sobolev space $W^{1,p}(0,1;E)$ is defined by equality \[W^{1,p}(0,1;E)=\{u\in L^p(I,E)\colon u'\in L^p(I,E)\}.\] It is a Banach space endowed with the norm $||u||_{W^{1,p}}:=||u||_p+||u'||_p$.
\par Given metric space X, a set-valued map $F\colon X\map E$ assigns to any $x \in X$ a nonempty subset $F(x)\subset E$. $F$ is (weakly) upper semicontinuous, if the small inverse image $F^{-1}(A)=\{x\in X\colon F(x)\subset A\}$ is open in $X$ whenever $A$ is (weakly) open in $E$. We have the following characterization: a map $F\colon X\map E$ with convex values is weakly upper semicontinues and has weakly compact values iff given a sequence $(x_n,y_n)$ in the graph $\Graph(F)$ with $x_n\to x$ in $X$, there is a subsequence $y_{k_n}\rightharpoonup y\in F(x)$ ($\rightharpoonup$ denotes the weak convergence). A map $F\colon X\map E$ is lower semicontinuous, if the large counter image $F^{-1}_+(A)=\{x\in X\colon F(x)\cap A\neq\varnothing\}$ is open in $X$ for any open $A\subset E$. The multimap $F\colon X\map E$ is proper if the preimage $F^{-1}_+(A)$ is compact for every compact subset $A$ of $E$. Let $\Omega$ be a set with a $\sigma$-field $\Sigma$ of subsets of $\Omega$. We say that $F\colon\Omega\map E$ is $\Sigma$-measurable iff for every open $A\subset E$ the large counter image $F^{-1}_+(A)\in\Sigma$. We shall call $F\colon I\times E\map E$ a lower Carath\'eodory multivalued map if $F(t,\cdot)$ is lower semicontinuous for each fixed $t\in I$ and the map $F(\cdot,\cdot)$ is ${\mathcal L}(I)\otimes{\mathcal B}(E)$-measurable, where ${\mathcal L}(I)$ and ${\mathcal B}(E)$ stands for Lebesgue $\sigma$-field of $I$ and Borel $\sigma$-field of $E$, respectively. Remind that the set-valued map $F\colon D\subset E\map E$ is said to be accretive if \[\langle u-w,x-y\rangle_+\geqslant 0\;\;\;\text{for all }x,y\in D, u\in F(x)\text{ and }w\in F(y),\] which will be abbreviated by $\langle F(x)-F(y),x-y\rangle_+\geqslant 0$ on $D\times D$. The set of all fixed points of the map $F\colon E\map E$ is denoted by $\fix(F)$.\par Let $H_*(\cdot)$ denote the \v Cech homology functor with coefficients in the field of rational numbers $\varmathbb{Q}$ (see \cite{andres,gorn}). A compact topological space $X$ having the property
\[H_q(X)=
\begin{cases}
0&\mbox{for }q\geqslant 1,\\
\varmathbb{Q}&\mbox{for }q=0
\end{cases}\]
is called acyclic. In other words its homology are exactly the same as the homology of a one point space. A compact (nonempty) space $X$ is an $R_\delta$-set if there is a decreasing sequence of contractible compacta $(X_n)_\n$ containing $X$ as a closed subspace such that $X=\bigcap_\n X_n$ (compare \cite{hyman}). In particular, $R_\delta$-sets are acyclic.\par An upper semicontinuous map $F\colon E\map E$ is called acyclic if it has compact acyclic values. A set-valued map $F\colon E\map E$ is admissible (compare \cite[Def.40.1]{gorn}) if there is a metric space $X$ and two continuous functions $p\colon X\to E$, $q\colon X\to E$ from which $p$ is a Vietoris map such that $F(x)=q(p^{-1}(x))$ for every $x\in E$. Clearly, every acyclic map is admissible. Moreover, the composition of admissible maps is admissible (\cite[Th.40.6]{gorn}). In particular the composition of two acyclic maps is admissible. 
\par A real function $\gamma$ defined on the family ${\mathcal B}(E)$ of bounded subsets of $E$ is called a measure of non-compactness (MNC) if $\gamma(\Omega)=\gamma(\overline{\co}\Omega)$ for any bounded subset $\Omega$ of $E$. The following example of MNC is of particular importance: given $E_0\subset E$ and $\Omega\in{\mathcal B}(E_0)$, 
\begin{align*}
\beta(\Omega;E_0):=\inf\Biggl\{\eps>0:\Biggr.&\text{ there are finitely many points }x_1,\dots,x_n\in E_0\\&\Biggl.\text{ with }\Omega\subset\bigcup_{i=1}^nB(x_i,\eps)\Biggr\}
\end{align*} 
is the Hausdorff MNC relative to the subspace $E_0$. Recall that this measure is regular, i.e. $\beta(\Omega;E_0)=0$ iff $\Omega$ is relatively compact in $E_0$; monotone, i.e. if $\Omega\subset\Omega'$ then $\beta(\Omega;E_0)\<\beta(\Omega';E_0)$ and invariant with respect to union with compact sets, i.e. $\beta(K\cup\Omega;E_0)=\beta(\Omega;E_0)$ for any relatively compact $K\subset E_0$ (for details see \cite{sadovski}). A set-valued map $F\colon E\map E$ is condensing relative to MNC $\gamma$ (or $\gamma$-condensing) provided, for every $\Omega\in{\mathcal B}(E)$, the set $F(\Omega)$ is bounded and $\gamma(\Omega)\<\gamma(F(\Omega))$ implies relative compactness of $\Omega$.\par Let $\varmathbb{R}_+^I$ and $\R{N}_+$ denote the partially ordered linear space of all scalar positively valued functions defined on $I$ and respectively the positive cone for the standard order on a vector lattice $\R{N}$. The following Darbo-Sadovskii-type fixed point theorem for condensing admissible maps settles the topological properties of the solution set of boundary value problems, which are the subject of interest in this paper.
\begin{theorem}\label{fix}
Let $\gamma\colon{\mathcal B}(E)\to\varmathbb{R}_+^I\times\R{N}_+$ be an MNC. Assume that $\gamma$ is
\begin{itemize}
\item[(i)] monotone, i.e. $\Omega_1\subset\Omega_2\Rightarrow\gamma(\Omega_1)\<\gamma(\Omega_2)$,
\item[(ii)] positively subhomogeneous, i.e. $\gamma(\lambdaup\Omega)\<\lambdaup\gamma(\Omega)\;\;$ for $\lambdaup\in[0,\infty)$,
\item[(iii)] algebraically semiadditive, i.e. $\gamma(\Omega_1+\Omega_2)\<\gamma(\Omega_1)+\gamma(\Omega_2)$,
\item[(iv)] regular, i.e. $\gamma(\Omega)=0\Leftrightarrow\Omega$ is relatively compact. 
\end{itemize}
Suppose that $D\subset E$ is nonempty closed convex and bounded and $F\colon D\map D$ is an admissible $\gamma$-condensing set-valued map. Then $\fix(F)$ is nonempty and compact.
\end{theorem}
\begin{proof}
Let us trace the scheme of proof of \cite[Th.59.12]{gorn}. In order to show \cite[Prop.59.3]{gorn} we need to know that the MNC $\gamma$ is monotone, algebraically semiadditive and semiregular in the sense that: $\Omega$ is relatively compact $\Rightarrow\gamma(\Omega)=0\in\varmathbb{R}_+^I\times\R{N}_+$. The map $F$ must be $\gamma$-condensing.\par The proof of \cite[Prop.59.11]{gorn} requires the presumption that $\gamma$ is regular and assumes its values in the partially ordered space $\varmathbb{R}_+^I$ or $\R{N}_+$ (or in the Cartesian product of these spaces).\par Eventually, the assumption that $\gamma$ is positively subhomogeneous and $F$ is $\gamma$-condensing enables us to carry out the proof of \cite[Th.59.12]{gorn}.\par In support of \cite[Prop.59.2]{gorn} we need to use the monotonicity of $\gamma$ and the assumption that $F$ is $\gamma$-condensing.
\end{proof}
\par The eponymous acyclicity of the solution set of boundary value problems under consideration is the result of application of the following multivalued generalization of Browder-Gupta theorem \cite[Th.2.1.]{gabor}.
\begin{theorem}\label{gabor}\cite[Th.2.16.]{gabor}
Let $X$ be a metric space, $E$ a Banach space and let $\Psi\colon X\map E$ be an usc proper set-valued map with compact values. Assume that there exists a sequence of compact convex valued usc proper multimaps $\Psi_n\colon X\map E$ such that
\begin{itemize}
\item[(i)] $\Psi_n(x)\subset B\f(\Psi\f(B\f(x,\frac{1}{n}\g)\g),\eps_n\g)$ for every $x\in X$, where $\eps_n\to 0^+$ as $n\to\infty$,
\item[(ii)] $0\in\Psi(x)\Rightarrow\Psi_n(x)\cap D(0,\eps_n)\neq\varnothing$,
\item[(iii)] for every $\n$ and every $u\in E$ with $|u|<\eps_n$ the set $\{x\in X\colon u\in\Psi_n(x)\}$ is nonempty acyclic.
\end{itemize}
Then the set $\,\Psi^{-1}(\{0\})$ is compact acyclic.
\end{theorem}
\section{ the Carath\'eodory case}
Denote by $L\colon D_L\to L^2(I,E)$ a continuous linear differential operator, given by the formula 
\[L:=\frac{d^2}{dt^2}+a_1(t)\frac{d}{dt}+a_0(t),\] where the domain $D_L$ forms a subspace of the Bochner-Sobolev space \[\Hi^2(0,1;E):=W^{2,2}(0,1;E):=\left\{u\in W^{1,2}(0,1;E)\colon u'\in W^{1,2}(0,1;E)\right\}\] accordingly to the definition \[D_L:=\left\{x\in\Hi^2(0,1;E)\colon x(0)=x(1)\text{ and }x'(0)=x'(1)\right\}.\]
An element $x\in\Hi^2(0,1;E)$ will be called a solution to \eqref{inclusion} if $x\in D_L$ and there is a square-integrable $w\in L^2(I,E)$ such that $w(t)\in F(t,x(t))$ for a.a. $t\in I$ and $Lx=w$. \par Recall that the Nemytski\v{\i} operator $N_F\colon\Hi^2(0,1;E)\map L^2(I,E)$, corresponding to the right-hand side $F$, is a set-valued map defined by \[N_F(x)=\left\{u\in L^2(I,E)\colon u(t)\in F(t,x(t))\mbox{ for a.a. }t\in I\right\}.\]
The BVP \eqref{inclusion} is equivalent to the following operator inclusion \[Lx\in N_F(x).\] 
\par Let us move on to the key issue of the assumptions, on which the results of the current section are based. We will use the following hypotheses on the mapping $F\colon I\times E\map E$:
\begin{itemize}
\item[$(\F_1)$] for every $(t,x)\in I\times E$ the set $F(t,x)$ is nonempty closed and convex,
\item[$(\F_2)$] the map $F(\cdot,x)$ has a strongly measurable selection for every $x\in E$,
\item[$(\F_3)$] the graph $\Graph(F(t,\cdot))$ is closed in $(E,|\cdot|)\times(E,\sigma(E,E^*))$ for a.a. $t\in I$,
\item[$(\F_4)$] $F$ possesses $L^2$-sublinear growth, i.e. there is $c\in L^2(I,\R{})$ and $m\geqslant 0$ such that for all $x\in E$ and for a.a. $t\in I$,\[||F(t,x)||^+:=\sup\{|y|\colon y\in F(t,x)\}\<c(t)+m|x|,\]
\item[$(\F_5)$] there is a function $\eta\in L^2(I,\R{})$ such that for all bounded subsets $\Omega\subset E$ and for a.a. $t\in I$ the inequality holds \[\beta(F(t,\Omega);E)\<\eta(t)\beta(\Omega;E),\]
\item[$(\F_6)$] the map $F(t,\cdot)$ is accretive for a.a. $t\in I$.
\end{itemize}
\begin{remark}\label{zalozenia}
Suppose that $E$ is reflexive. By $(\F_4)$ the map $F(t,\cdot)$ is locally bounded a.e. on $I$. Consider a sequence $(x_n,y_n)_\n$ in the graph $\Graph(F(t,\cdot))$ with $x_n\to x$ in the norm of $E$. Since $E$ is reflexive and $F(t,\cdot)$ locally bounded, there must be a subsequence $y_{k_n}\rightharpoonup y$. Now, bearing in mind $(\F_3)$, i.e. that $\Graph(F(t,\cdot))$ is strongly-weakly closed, we obtain $y\in F(t,x)$. Therefore, $F(t,\cdot)$ is weakly upper semicontinuous for a.a. $t\in I$.\par In general, weak upper semicontinuity is a significantly stronger assumption then the condition $(\F_3)$.
\end{remark} 
\par The following assumption is our standing hypothesis for the rest of the ongoing section:\par\vspace{0.5\baselineskip}\noindent
{\bf Assumption $\mathbf{(G_1)}$:}
\begin{minipage}[t]{9.5cm}
The coefficient mappings $a_1\colon I\to\R{}$ and $a_0\colon I\to(-\infty,0]$ are continuous. The reduced system (the scalar completely homogeneous boundary value problem) 
\begin{equation}\label{equation}
\begin{gathered}
x''+a_1(t)x'+a_0(t)x=0,\\
x(0)=x(1),\\
x'(0)=x'(1)
\end{gathered}
\end{equation}
is incompatible, i.e. possesses only the trivial solution. 
\end{minipage}
\begin{remark}
Consider the boundary conditions operators $B_1, B_2\colon C^1(I,\R{})\to\R{}$ of the form \[B_1:=\ev_0-\ev_1\;\;\text{and}\;\;B_2:=\ev_0\circ\frac{d}{dt}-\ev_1\circ\frac{d}{dt}.\] Assumption $\mathbf{(G_1)}$ holds iff 
\[\begin{vmatrix}
B_1u_1 & B_1u_2\\
B_2u_1 & B_2u_2
\end{vmatrix}\neq 0,\]
where $u_1,u_2$ is a fundamental system of solutions of a homogeneous linear differential equation: $x''+a_1(t)x'+a_0(t)x=0$ on $I$.
\end{remark}
If assumption $\mathbf{(G_1)}$ is satisfied, then  there exists (only one) so-called influence function $G\colon I\times I\to\R{}$ for the problem \eqref{equation}, in which case the mapping $x\in D_L$, given by \[x(t)=\int_0^1G(t,s)u(s)\,ds\;\;\text{for }t\in I,\] provides a unique solution to the inhomogeneous problem $Lx=u\in L^2(I,E)$. This means that the set $S_{\!F}$ of solutions to periodic problem \eqref{inclusion} coincides with the solution set to the following Hammerstein integral inclusion
\begin{equation}\label{inclusion2}
x(t)\in\int_0^1G(t,s)F(s,x(s))\,ds,\;\;t\in I.
\end{equation}
Denote by $H\colon L^2(I,E)\to\Hi^2(0,1;E)$ the associated Hammerstein integral operator:
\begin{equation}\label{hammerstein}
H(u)(t)=\int_0^1 G(t,s)u(s)ds,\;t\in I.
\end{equation}
\indent We are now in position to state and prove our first main result in the Carath\'eodory case. It provides an insight into the topological structure of the solutions set $S_{\!F}$ of the periodic problem \eqref{inclusion}.
\begin{theorem}\label{solset}
Let $E$ be a reflexive Banach space. Suppose that the multimap $F\colon I\times E\map E$ satisfies assumptions $(\F_1)$-$(\F_5)$. Assume that the spectral radius $r(\hat{H})$ of the related linear operator $\hat{H}\colon C(I,\R{})\to C(I,\R{})$ is less than $1$, where $\hat{H}$ is defined by
\begin{equation}\label{operator}
\hat{H}(u):=2\int_0^1|G(\cdot,s)|\eta(s)u(s)\,ds.
\end{equation}
Suppose further that the right-hand side $F$ is $L^2$-integrably bounded or the constant $m$ in $(\F_4)$ is strictly positive and the following inequality is met
\begin{equation}\label{bound}
m\sup\limits_{t\in I}||G(t,\cdot)||_2<1.
\end{equation}
Then the periodic boundary value problem \eqref{inclusion} possesses a solution. Moreover, the solution set $S_{\!F}$ to problem \eqref{inclusion} is compact in the space $C^1(I,E)$ and weakly compact as a subset of the space $\Hi^2(0,1;E)$.
\end{theorem}
\begin{proof}
Observe that the solution set of integral inclusion \eqref{inclusion2} corresponds to the set $\fix(\Phi)$ of fixed points of the operator $\Phi\colon C^1(I,E)\map C^1(I,E)$, given by \[\Phi(x):=(H\circ N_F)(x).\] We claim that the multivalued operator $\Phi$ is upper semicontinuous and possesses {non\-empty} compact convex values. Precisely, we will show that if $x_n\to x$ in $C^1(I,E)$ and $y_n\in \Phi(x_n)$, then there is a subsequence $(y_{k_n})_\n$ convergent in $C^1(I,E)$ to $y\in\Phi(x)$. From \cite[Prop.1.]{pietkun} and the fact that $H$ is linear it follows that $\Phi(x)$ is nonempty convex for every $x\in C^1(I,E)$. Let $x_n\to x$ in $C^1(I,E)$ and $y_n\in\Phi(x_n)$. Then $y_n=H(w_n)$ for some $w_n\in N_F(x_n)$. Since the operator $N_F$ is weakly upper semicontinuous (remind yourself Remark \ref{zalozenia}. and apply it in the context of \cite[Prop.1.]{pietkun}), there is a subsequence (again denoted by) $(w_n)_\n$ such that $w_n\rightharpoonup w\in N_F(x)$. \par Let us introduce an auxiliary notation: $\bigtriangleup:=\{(t,s)\in I\times I\colon 0\<s\<t\<1\}$ and $\bigtriangledown:=\{(t,s)\in I\times I\colon 0\<t\<s\<1\}$. Remember that the Green's function $G$ has partial derivatives in $t$ of the order $\<2$ and these derivatives are continuous in each triangle $\bigtriangleup$ and $\bigtriangledown$. Now, fix $\tau\in I$. It is easy to see that
\[\left|\frac{\partial}{\partial t}G(t,s)-\frac{\partial}{\partial t}G(\tau,s)\right|^2\xrightarrow[t\to\tau]{}0\] for every $s\in I$ and
\[\left|\frac{\partial}{\partial t}G(t,s)-\frac{\partial}{\partial t}G(\tau,s)\right|^2\<\left(2\sup_{(\zeta,\xi)\in\bigtriangleup}\left|\frac{\partial}{\partial t}G(\zeta,\xi)\right|+2\sup_{(\zeta,\xi)\in\bigtriangledown}\left|\frac{\partial}{\partial t}G(\zeta,\xi)\right|\right)^2\] for any $(t,s)\in I^2$. This means that the function $I\ni t\mapsto\frac{\partial}{\partial t}G(t,\cdot)\in L^2(I)$ is uniformly continuous. Clearly, the mapping $I\ni t\mapsto G(t,\cdot)\in L^2(I)$ is continuous as well.\par The image $\Phi(\Omega)$ and the derivative $\Phi(\Omega)'$ of this set are equicontinuous for any bounded $\Omega\subset C^1(I,E)$. It follows by the estimations
\begin{align*}
\sup_\n|H(w_n)(t)-H(w_n)(\tau)|&\<\sup_\n\int_0^1|G(t,s)-G(\tau,s)||w_n(s)|\,ds\\&\<||G(t,\cdot)-G(\tau,\cdot)||_2\f(||c||_2+m\sup_\n||x_n||\g)
\end{align*}
and \[\sup_\n\f|\frac{d}{dt}H(w_n)(t)-\frac{d}{dt}H(w_n)(\tau)\g|\<\f\Arrowvert\frac{\partial}{\partial t}G(t,\cdot)-\frac{\partial}{\partial t}G(\tau,\cdot)\g\Arrowvert_2\f(||c||_2+m\sup_\n||x_n||\g).\] In view of the Pettis measurability theorem there exists a closed linear separable subspace $E_t$ of $E$ such that \[\left\{\int\limits_0^1G(t,s)w_n(s)\,ds,\int\limits_0^1\frac{\partial}{\partial t}G(t,s)w_n(s)\,ds,\frac{\partial}{\partial t}G(t,s)w_n(s),G(t,s)w_n(s)\right\}_{n\geqslant 1}\subset E_t\] for a.a. $s\in I$. Under assumption $(\F_5)$ the following inequality is satisfied: 
\begin{align*}
\beta(G(t,s)\{w_n(s)\}_{n\geqslant 1};E_t)&\<\beta(G(t,s)\{w_n(s)\}_{n\geqslant 1};E)=|G(t,s)|\beta(\{w_n(s)\}_{n\geqslant 1};E)\\&\<|G(t,s)|\beta(F(s,\{x_n(s)\}_{n\geqslant 1});E)\<|G(t,s)|\eta(s)\beta(\{x_n(s)\}_{n\geqslant 1};E).
\end{align*}
Applying the latter in the context of \cite[Cor.3.1]{heinz} one obtains, for every $t\in I$
\begin{equation}\label{osz1}
\begin{split}
\beta(\{y_n(t)\}_{n\geqslant 1};E_t)&=\beta\left(\left\{\int_0^1G(t,s)w_n(s)\,ds\right\}_{n\geqslant 1};E_t\right)\<\int_0^1\beta\left(G(t,s)\{w_n(s)\}_{n\geqslant 1};E_t\right)\,ds\\&\<\int_0^1|G(t,s)|\eta(s)\beta\left(\{x_n(s)\}_{n\geqslant 1};E\right)\,ds
\end{split}
\end{equation}
and 
\begin{equation}\label{osz2}
\begin{split}
\beta\left(\{\dot{y}_n(t)\}_{n\geqslant 1};E_t\right)\<\int_0^1\left|\frac{\partial}{\partial t}G(t,s)\right|\eta(s)\beta\left(\{x_n(s)\}_{n\geqslant 1};E\right)\,ds.
\end{split}
\end{equation}
Since $x_n\rightrightarrows x$ on $I$ we infer that $\beta(\{y_n(t)\}_{n\geqslant 1};E_t)=0$ and $\beta(\{y'_n(t)\}_{n\geqslant 1};E_t)=0$ for every $t\in I$. In view of the Ascoli-Arzel\`a theorem the sequence $(y_n)_{n\geqslant 1}$ possesses a subsequence convergent to some function $y$ in the norm of $C^1(I,E)$.\par It is a matter of routine to check that the mapping $H$ is continuous as an operator from $L^2(I,E)$ to $\Hi^2(0,1;E)$. This follows immediately from the estimates:
\begin{align*}
\f\Arrowvert\frac{d^2}{dt^2}H(w_n)-\frac{d^2}{dt^2}H(w)\g\Arrowvert_2&=\f\Arrowvert w_n-w+\int_0^1\frac{\partial^2}{\partial t^2}G(\cdot,s)(w_n(s)-w(s))\,ds\g\Arrowvert_2\\&\<||w_n-w||_2+\f(\int_0^1\f|\int_0^1\frac{\partial^2}{\partial t^2}G(t,s)(w_n(s)-w(s))\,ds\g|^2dt\g)^\frac{1}{2}\\&\<||w_n-w||_2+\f(\int_0^1\f(\int_0^1\f|\frac{\partial^2}{\partial t^2}G(t,s)\g||w_n(s)-w(s)|\,ds\g)^2dt\g)^\frac{1}{2}\\&\<||w_n-w||_2+\f(\int_0^1\f\Arrowvert\frac{\partial^2}{\partial t^2}G(t,\cdot)\g\Arrowvert_2^2||w_n-w||_2^2\,dt\g)^\frac{1}{2}\\&\<||w_n-w||_2\f(\f\Arrowvert\,\f\Arrowvert\frac{\partial^2}{\partial t^2}G(t,\cdot)\g\Arrowvert_2\g\Arrowvert_2+1\g),
\end{align*}
\[\f\Arrowvert\frac{d}{dt}H(w_n)-\frac{d}{dt}H(w)\g\Arrowvert_2\<||w_n-w||_2\f\Arrowvert\,\f\Arrowvert\frac{\partial}{\partial t}G(t,\cdot)\g\Arrowvert_2\g\Arrowvert_2\] and \[||H(w_n)-H(w)||_2\<||w_n-w||_2||\,||G(t,\cdot)||_2||_2.\]
Recall that we have already established: $w_n\rightharpoonup w\in N_F(x)$. Since $H$ is a linear operator, we see that $H(w_n)\rightharpoonup H(w)$ in $C^1(I,E)$. Therefore, $y=H(w)$ and eventually $y\in\mf(x)$. As we have found $\Phi$ is a compact convex valued usc multimap.\par It is worthwhile to observe that instead of the assumption $(\F_4)$ we may use w.l.o.g the following property:
\begin{itemize}
\item[$(\F_4')$] $||F(t,x)||^+\<\hat{c}(t)$ for every $x\in E$ and for a.a. $t\in I$, where $\hat{c}\in L^2(I,\R{}_+)$.
\end{itemize}
Indeed, if $x\in S{\!_F}$, then
\begin{align*}
||x||&\<\sup_{t\in I}\int_0^1|G(t,s)|(c(s)+m|x(s)|)\,ds\<\sup_{t\in I}||G(t,\cdot)||_2(||c||_2+m||x||_2)\\&=\sup_{t\in I}||G(t,\cdot)||_2||c||_2+m\sup_{t\in I}||G(t,\cdot)||_2||x||
\end{align*}
Thus, if \eqref{bound} holds, then 
\[||x||\<\frac{\sup_{t\in I}||G(t,\cdot)||_2||c||_2}{1-m\sup_{t\in I}||G(t,\cdot)||_2},\]
which means that the solution set $S_{\!F}$ is bounded as a subset of $C(I,E)$. Now, if we denote by $r\colon E\to D(0,M)$ the radial retraction onto the closed ball $D_C(0,M)$ such that $S_{\!F}\subset D_C(0,M)\subset C(I,E)$, then the solution set $S_{\!\hat{F}}$ to the integral inclusion
\[x\in(H\circ N_{\!\hat{F}})(x),\]
where the set-valued map $\hat{F}\colon I\times E\map E$ is such that $\hat{F}(t,x)=F(t,r(x))$, coincides with the set $S_{\!F}$. Evidently, the map $\hat{F}$ satisfies assumptions $(\F_1)$, $(\F_2)$ and $(\F_4')$ with $\hat{c}(\cdot):=c(\cdot)+mM$. Note that the radial retraction is Lipschitz and $1$-$\beta$-contractive. Therefore, the map $\hat{F}$ satisfies both condition $(\F_3)$ and $(\F_5)$ (in fact, $\hat{F}$ is weakly upper semicontinuous).\par Notice that the operator $\mf$ is bounded. Actually, if $y\in\mf(x)$ for some $x\in C^1(I,E)$, then
\begin{equation}\label{intbound}
\begin{split}
||y||_{C^1}&=||y||+||\dot{y}||\<R:=||\hat{c}||_2\f(\sup_{t\in I}||G(t,\cdot)||_2+\sup_{t\in I}\f\Arrowvert\frac{\partial}{\partial t}G(t,\cdot)\g\Arrowvert_2\g),
\end{split}
\end{equation}
where $\hat{c}$ is the integral bound of $F$ under the assumption $(\F_4')$. Therefore the inclusion $\mf\f(D_{C^1}(0,R)\g)\subset D_{C^1}(0,R)\subset C^1(I,E)$ does not require a comment. \par Recall that for $\Omega$ bounded in $C(I,E)$ the expression \[[\psi_C^1(\Omega)](t):=\beta(\Omega(t);E),\] where $\Omega(t):=\{x(t)\colon x\in\Omega\}$, defines a MNC on the space $C(I,E)$ (\cite[Ex.1.2.4.]{sadovski}). In this space the formula of the modulus of equicontinuity of the set of functions $\Omega\subset C(I,E)$ has the following form \[\modulus_C(\Omega):=\lim\limits_{\delta\to 0^+}\sup\limits_{x\in\Omega}\max\limits_{|t-\tau|\<\delta}|x(t)-x(\tau)|.\] It defines a MNC on $C(I,E)$ as well (\cite[Ex.2.1.2]{obu}). Furthermore, if $\modulus_C(\Omega)=0$, then $\psi_C^1(\Omega)\in C(I,\R{}_+)$. Let us define a set function $\psi_{C^1}\colon{\mathcal B}\f(C^1(I,E)\g)\to\varmathbb{R}^I_+\times\varmathbb{R}^I_+\times\R{2}_+$ in the following way
\begin{equation}\label{MNC}
\psi_{C^1}(\Omega):=\max_{D\in\Delta(\Omega)}\f(\psi_C^1(D),\psi_C^1(D'),\modulus_C(D),\modulus_C(D')\g),
\end{equation}
where $D':=\{x'\colon x\in D\}$, $\Delta(\Omega)$ stands for the family of countable subsets of $\Omega$ and the product $\varmathbb{R}^I_+\times\varmathbb{R}_+^I\times\R{2}_+$ is provided with the usual point- and component-wise order relation. It is a matter of routine to check that $\psi_{C^1}$ is a monotone, positively homogeneous, algebraically semiadditive and invariant with respect to union with compact sets MNC on the space $C^1(I,E)$. What is most important, this measure is regular, which follows directly from Ascoli-Arzel\`a theorem. We claim that the set-valued operator $\mf\colon D(0,R)\map D(0,R)$ is condensing relative to MNC $\psi_{C^1}$.\par Let $\Omega$ be a bounded subset of $C^1(I,E)$ for which the inequality holds
\begin{equation}\label{cond}
\psi_{C^1}(\Omega)\<\psi_{C^1}(\mf(\Omega)).
\end{equation}
Suppose that quantities $\psi_{C^1}(\Omega)$ and $\psi_{C^1}(\mf(\Omega))$ are attained respectively on denumerable subsets $\{z_n\}_\n$ and $\{y_n\}_\n$. There are also functions $x_n\in\Omega$ and $w_n\in N_F(x_n)$ such that $y_n=H(w_n)$. In accordance with \eqref{osz1}, \eqref{osz2} and \eqref{cond}, we have
\begin{equation}\label{ko}
\begin{aligned}
\f[\psi_C^1\f(\{z_n\}_{n\geqslant 1}\g)\g](t)&\!\<\!\f[\psi_C^1\f(\{y_n\}_{n\geqslant 1}\g)\g](t)\!\<\!2\beta(\{y_n(t)\}_{n\geqslant 1};E_t)\<2\!\!\int\limits_0^1\!|G(t,s)|\eta(s)\beta(\{x_n(s)\}_{n\geqslant 1};E)\,ds\\&\<2\int_0^1|G(t,s)|\eta(s)\f[\psi_C^1\f(\{z_n\}_{n\geqslant 1}\g)\g](s)\,ds
\end{aligned}
\end{equation}
and
\begin{equation}\label{ko2}
\begin{split}
\f[\psi_C^1\f(\{z'_n\}_{n\geqslant 1}\g)\g](t)&\<\f[\psi_C^1\f(\{y'_n\}_{n\geqslant 1}\g)\g](t)\<2\beta(\{\dot{y}_n(t)\}_{n\geqslant 1};E_t)\\&\<2\int_0^1\f|\frac{\partial}{\partial t}G(t,s)\g|\eta(s)\beta(\{x_n(s)\}_{n\geqslant 1};E)\,ds\\&\<2\int_0^1\f|\frac{\partial}{\partial t}G(t,s)\g|\eta(s)\f[\psi_C^1\f(\{z_n\}_{n\geqslant 1}\g)\g](s)\,ds.
\end{split}
\end{equation}
On the other hand it is true that \[\modulus_C\f(\{z_n\}_\n\g)\<\modulus_C\f(\{y_n\}_\n\g)=0\;\;\text{and}\;\;\modulus_C\f(\{z'_n\}_\n\g)\<\modulus_C\f(\{y'_n\}_\n\g)=0.\]
Thus, $\psi_C^1\f(\{z_n\}_\n\g)\in C(I,\R{}_+)$ and $\f(id-\hat{H}\g)\f(\psi_C^1\f(\{z_n\}_\n\g)\g)\<0$, by \eqref{ko}. Since the spectral radius $r(\hat{H})<1$, we infer that $\psi_C^1\f(\{z_n\}_\n\g)\<\f(id-\hat{H}\g)^{-1}(0)=0$ (more details the reader may find, for example in \cite[Prop.1.2]{monch} or in \cite[p.168-169]{deimling2}). Therefore, \eqref{ko2} entails $\psi_C^1\f(\{z'_n\}_\n\g)=0$. Eventually, $\psi_{C^1}\f(\{z_n\}_\n\g)=0\in\varmathbb{R}^I_+\times\varmathbb{R}_+^I\times\R{2}_+$, which means that $\Omega$ is relatively compact in $C^1(I,E)$ and $\mf$ is a $\psi_{C^1}$-condensing operator.\par The preceding considerations indicate that the operator $\mf\colon D_{C^1}(0,R)\map D_{C^1}(0,R)$ is an admissible $\psi_{C^1}$-condensing set-valued map, allowing us to use Theorem \ref{fix}. Consequently, the solution set $S_{\!F}$ is nonempty and compact in the space $C^1(I,E)$.\par The weak compactness of $\fix(\Phi)$ in the space $\Hi^2(0,1;E)$ is not particularly sophisticated issue. If $(x_n)_\n$ is a sequence of fixed points of the operator $\Phi$, then we know already that, passing to a subsequence if necessary, $x_n\to x$ in $C^1(I,E)$. On the other hand, $w_n\rightharpoonup w$ in $L^2(I,E)$, where $Lx_n=w_n\in N_F(x_n)$. Furthermore, $w\in N_F(x)$. As we have shown the Hammerstein operator $H$ is continuous. That's why $H(w_n)\rightharpoonup H(w)$ in $\Hi^2(0,1;E)$. At the same time $x_n=H(w_n)\rightharpoonup H(w)$ and $x_n\rightharpoonup x$ in $C^1(I,E)$. Therefore, $x_n\rightharpoonup x$ in $\Hi^2(0,1;E)$, where $x=H(w)\in H\circ N_F(x)$, i.e. $x\in\fix(\Phi)$.    
\end{proof}
\begin{remark}\label{rem1}
Assumptions of Theorem \ref{solset}. allow to reformulate its thesis in the context of the problem \eqref{inclusion2}. Introducing minor adjustments in presented above reasoning one can easily show that the integral inclusion
\[x(t)\in h(t)+\int_0^1G(t,s)F(s,x(s))\,ds,\;\;\;t\in I\] possesses at least one solution for each fixed inhomogeneity $h\in C^1(I,E)$.
\end{remark}
\begin{remark}\label{rem2}
Instead of assuming that inequality $r(\hat{H})<1$ holds, we might as well assume that the following condition is satisfied$:$
\begin{equation}\label{nierzsup}
2\sup\limits_{t\in I}||G(t,\cdot)||_2||\eta||_2<1.
\end{equation}
Assumption \eqref{nierzsup} allows us to use \cite[Th.10.]{pietkun}. In view of this result, there exists a continuous solution to the integral inclusion \eqref{inclusion2}. Properties of the integral kernel $G$ imply affiliation of this solution to the subspace $\Hi^2(0,1;E)$. Therefore, the non-emptiness of the solution set $S_{\!F}$ of the problem \eqref{inclusion} also follows from condition \eqref{nierzsup}, although the use of this assumption does not strengthen the thesis of Theorem \ref{solset}. $($as indicated by Example \ref{Ex}.$)$.
\end{remark}
\begin{example}\label{Ex}
Consider \eqref{inclusion} with $a_1\equiv 0$ and $a_0\equiv -1$. The Green's function for the completely homogeneous boundary value problem
\begin{equation}\label{reduced}
\begin{gathered}
x''-x=0,\\
x(0)=x(1),\\
x'(0)=x'(1).
\end{gathered}
\end{equation}
has the following form
\begin{equation}\label{green}
G(t,s):=
\begin{cases}
{\displaystyle \frac{1}{2}\f(\frac{1}{1-e}e^{s-t}+\frac{e}{1-e}e^{t-s}\right)}&\text{for }\,0\<t\<s\<1,\\
\mbox{}&\\
{\displaystyle \frac{1}{2}\f(\frac{e}{1-e}e^{s-t}+\frac{1}{1-e}e^{t-s}\right)}&\text{for }\,0\<s\<t\<1.\\
\end{cases}
\end{equation}
The linear operator $\hat{H}$ from \eqref{operator} is compact and maps the convex cone $K:=C(I,\R{}_+)$ into itself. Obviously, $K-K=C(I,\R{})$. In this case, the Krein-Rutman theorem asserts that $r(\hat{H})=\max\{|\lambdaup|\colon\lambdaup\in\sigma_p(\hat{H})\}$, where $\sigma_p(\hat{H})$ is the point spectrum of $\hat{H}$. The equation $\hat{H}(u)=\lambdaup u$ is equivalent to $u=H\f(-\frac{2}{\lambdaup}\eta u\g)$, i.e.
\[u(t)=-\frac{2}{\lambdaup}\int_0^1G(t,s)\eta(s)u(s)\,ds,\;\;\;t\in I.\] The latter means that
\begin{equation}\label{przyklad3}
\begin{gathered}
u''-u=-\frac{2}{\lambdaup}\eta u,\\
u(0)=u(1),\\
u'(0)=u'(1).
\end{gathered}
\end{equation}
The principal assertion of the spectral theory of the periodic Sturm-Liouville problems says that $\lambdaup^{-1}$ is an eigenvalue associated with \eqref{przyklad3} iff $\lambdaup^{-1}$ is a root of $D(\lambdaup^{-1}):=u_1(1)+u'_2(1)=2$, where $D(\lambdaup)$ is the so-called Hill discriminant and $u_1,u_2$ are two independent solutions of differential equation $u''-u=-\frac{2}{\lambdaup}\eta u$ with initial conditions
\[\begin{cases}
u_1(0)=1&\\
u_1'(0)=0&
\end{cases}\;\;\;\text{and}\;\;\;
\begin{cases}
u_2(0)=0&\\
u_2'(0)=1&
\end{cases}\]
$($see \cite{eastham}$)$. In case $\eta$ is constant the problem \eqref{przyklad3} has a nontrivial solution only if $\lambdaup^{-1}=\frac{1}{2\eta}$. Hence, $r(\hat{H})=2\eta$. Therefore, the inequality $r(\hat{H})<1$ is met iff $2\eta<1$. On the other hand, condition \eqref{nierzsup} assumes the form $2\sup_{t\in I}||G(t,\cdot)||_2\eta<1$. A straightforward calculation shows that
\begin{align*}
\left(\sup_{t\in I}||G(t,\cdot)||_2\right)^2&=\frac{1}{2}\sup_{t\in I}\left(\frac{1}{4}\frac{e^2}{(1-e)^2}(1-e^{-2t})-\frac{1}{4}\frac{1}{(1-e)^2}(1-e^{2t})+\frac{e}{(1-e)^2}t\g.\\&+\f.\frac{1}{4}\frac{1}{(1-e)^2}(e^{2-2t}-1)-\frac{1}{4}\frac{e^2}{(1-e)^2}(e^{-2+2t}-1)+\frac{e}{(1-e)^2}(1-t)\g)\\&=\frac{e^2+2e-1}{4(e-1)^2},
\end{align*}
i.e. $\sup_{t\in I}||G(t,\cdot)||_2\approx 1.00066$. This leads to the conclusion that assumption $r(\hat{H})<1$ is slightly subtler than condition \eqref{nierzsup}. This should not be surprising, given that $r(\hat{H})\<||\hat{H}||_{\mathcal L}$.
\end{example}

\begin{example}
Consider the following inhomogeneous linear periodic boundary value problem:
\begin{equation}\label{przyklad}
\begin{gathered}
x''-x'-x\in F(t,x),\;\text{a.e. on } I,\\
x(0)=x(1),\\
x'(0)=x'(1).
\end{gathered}
\end{equation}
Since the respective completely homogeneous boundary value problem possesses only the trivial solution, the periodic problem \eqref{przyklad} is equivalent to the integral inclusion \[x(t)\in\int_0^1G(t,s)F(s,x(s))\,ds,\;\;t\in I,\] where the Green's function $G\colon I^2\to\R{}$ is given by
\begin{equation*}
G(t,s):=
\begin{cases}
{\displaystyle \frac{1}{\lambdaup_2-\lambdaup_1}\left(\frac{-e^{\lambdaup_1}}{1-e^{\lambdaup_1}}e^{\lambdaup_1(t-s)}+\frac{e^{\lambdaup_2}}{1-e^{\lambdaup_2}}e^{\lambdaup_2(t-s)}\right)}&\text{for }\,0\<t\<s\<1,\\
\mbox{}&\\
{\displaystyle\frac{1}{\lambdaup_2-\lambdaup_1}\left(\frac{-1}{1-e^{\lambdaup_1}}e^{\lambdaup_1(t-s)}+\frac{1}{1-e^{\lambdaup_2}}e^{\lambdaup_2(t-s)}\right)}&\text{for }\,0\<s\<t\<1.\\
\end{cases}
\end{equation*}
and $\lambdaup_1$, $\lambdaup_2$ are the roots of the respective characteristic equation, i.e. $\lambdaup_1:=\frac{1}{2}(\sqrt{5}+1)$, $\lambdaup_2:=-\frac{1}{2}(\sqrt{5}-1)$. It is not difficult to calculate that
\begin{align*}
\left(\sup_{t\in I}||G(t,\cdot)||_2\right)^2&=\sup_{t\in I}||G(t,\cdot)||_2^2=\sup_{t\in I}\left(\int_0^t|G(t,s)|^2\,ds+\int_t^1|G(t,s)|^2\,ds\right)\\&=\frac{1}{5}\sup_{t\in I}\left(\frac{2}{(1-e^{\lambdaup_1})(1-e^{\lambdaup_2})}\left(1-e^{(\lambdaup_1+\lambdaup_2)t}\right)+\frac{1}{2\lambdaup_1}\left(\frac{1}{1-e^{\lambdaup_1}}\right)^2\left(e^{2\lambdaup_1t}-1\right)\right.\\&+\frac{1}{2\lambdaup_2}\left(\frac{1}{1-e^{\lambdaup_2}}\right)^2\left(e^{2\lambdaup_2t}-1\right)+\frac{2e^{\lambdaup_1}e^{\lambdaup_2}}{(1-e^{\lambdaup_1})(1-e^{\lambdaup_2})}\left(e^{(\lambdaup_1+\lambdaup_2)(t-1)}-1\right)\\&\left.+\frac{1}{2\lambdaup_1}\left(\frac{e^{\lambdaup_1}}{1-e^{\lambdaup_1}}\right)^2\left(1-e^{2\lambdaup_1(t-1)}\right)+\frac{1}{2\lambdaup_2}\left(\frac{e^{\lambdaup_2}}{1-e^{\lambdaup_2}}\right)^2\left(1-e^{2\lambdaup_2(t-1)}\right)\right)\\&\approx 1.0013
\end{align*}
Consequently, $\sup_{t\in I}||G(t,\cdot)||_2\approx 1.00065$. In this way one obtains an upper bound for constant $m$ appearing in assumption \eqref{bound}{\em :} \[m<\left(\sup_{t\in I}||G(t,\cdot)||_2\right)^{-1}\approx 0.999\]
\end{example}

\begin{corollary}
Let $F\colon I\times\R{N}\map\R{N}$ satisfy $(\F_1)$-$(\F_4)$. Assume also that \eqref{bound} holds. Then the thesis of Theorem \ref{solset}. remains in force. Moreover, the set-valued map $\mf\colon\Hi^2\f(0,1;\R{N}\g)\map\Hi^2\f(0,1;\R{N}\g)$ is weakly sequentially to weakly sequentially upper semicontinuous. \end{corollary}
\begin{proof}
A commentary requires only the continuity of the operator $\Phi$. Suppose that $x_n\rightharpoonup x$ in $\Hi^2\f(0,1;\R{N}\g)$. Since $W^{1,2}\f(0,1;\R{N}\g)^*\subset W^{2,2}\f(0,1;\R{N}\g)^*$, the sequence $(x_n)_\n$ converges weakly in $W^{1,2}\f(0,1;\R{N}\g)$ ass well. We shall invoke the following convenient property of the weak convergence in the Sobolev space $W^{1,2}\f(0,1;\R{N}\g)$: \[x_n\rightharpoonup x\;\;\text{in }W^{1,2}\f(0,1;\R{N}\g)\Rightarrow 
\begin{cases}
x_n\to x\;\;\text{in }L^1\f(I,\R{N}\g),&\\
\dot{x}_n\rightharpoonup\dot{x}\;\;\text{in }L^2\f(I,\R{N}\g).
\end{cases}\]
Hence, passing to a subsequence if necessary, we may assume that $x_n(t)\to x(t)$ a.e. on $I$. Now, take $y_n\in\Phi(x_n)$. Then $y_n=H(w_n)$ for some $w_n\in N_F(x_n)$. In view of the Eberlein-\v{S}mulian theorem, there is a subsequence (again denoted by) $(w_n)_\n$ converging weakly to some $w$ in $L^2(I;\R{N})$. By the Convergence Theorem (see \cite[Th.2.]{pietkun}) we get $w(t)\in F(t,x(t))$ a.e. on $I$. Hence, $y_n=H(w_n)\rightharpoonup y=H(w)\in\Phi(x)$ in $\Hi^2\f(0,1;\R{N}\g)$. This means that the preimage $\Phi^{-1}(A)$ is weakly sequentially closed for all weakly sequentially closed $A\subset\Hi^2\f(0,1;\R{N}\g)$.
\end{proof}
\begin{corollary}\label{solset2}
Let $E$ be a separable Banach space. Suppose that the closed valued multimap $F\colon I\times E\map E$ is lower Carath\'eodory and satisfies assumptions $(\F_4)$-$(\F_5)$. Suppose further that the spectral radius $r(\hat{H})$ of the linear operator $\hat{H}$ from \eqref{operator} is less than $1$ and the inequality \eqref{bound} holds. Then the periodic boundary value problem \eqref{inclusion} possesses at least one solution.
\end{corollary}
\begin{proof}
Denote by $S(I,E)$ the space of all $({\mathcal L}(I),{\mathcal B}(E))$-measurable functions mapping $I$ to $E$, equiped with the topology of convergence in measure. By virtue of \cite[p.731]{ponosov} the Nemytski\v{\i} operator $N_F\colon S(I,E)\map S(I,E)$ is lower semicontinuous. Consider a sequence $(x_n)_\n$ such that $||x_n-x||_{C^1}\to 0$. In particular, $x_n\to x$ in measure. If $y$ is an arbitrary element of $N_F(x)$, then there exists a sequence $y_n\in N_F(x_n)$ such that $y_n\to y$ in $S(I,E)$. From every subsequence of $(y_n)_\n$ we can extract some subsequence $(y_{k_n})_\n$ satisfying $y_{k_n}(t)\to y(t)$ a.e. on $I$. This subsequence must be $L^2$-integrably bounded in view of $(\F_4)$. Thus, $y_{k_n}\to y$ in $L^2(I,E)$. Bearing in mind that $S(I,E)$ is metrizable, we infer that $||y_n-y||_2\to 0$. Therefore, the multimap $N_F\colon C^1(I,E)\map L^2(I,E)$ meets the definition of lower semicontinuity. It is also clear that this operator possesses closed and decomposable values. In view of \cite[Th.3.]{brescol} there exists a continuous $n_F\colon C^1(I,E)\to L^2(I,E)$ such that $n_F(x)\in N_F(x)$ for every $x\in C^1(I,E)$. In particular, for every bounded $\Omega\subset C^1(I,E)$ and for a.a. $t\in I$ we have $\beta\f(n_F(\Omega)(t)\g)\<\eta(t)\beta(\Omega(t))$. Let $R>0$ be given by \eqref{intbound}. One can easily see that the operator $H\circ n_F\colon D_{C^1}(0,R)\to D_{C^1}(0,R)$ is continuous and condensing relative to MNC $\psi_{C^1}$ (the justification is completely analogous to the arguments contained in the proof of Theorem \ref{solset}.). Consequently, this operator possesses a fixed point. Since $\fix(H\circ n_F)\subset\fix(H\circ N_F)$, the solution set $S_{\!F}$ must be nonempty.
\end{proof}
\indent In order to show that the solution set $S_{\!F}$ possesses acyclic structure we shall use the trick based on the strong accretivity of the sum $-L+N_{F_n}$, where $N_{F_n}$ is the Nemytski\v{\i} operator induced by an appropriately chosen approximation $F_n$ of the right-hand side of the differential inclusion. Therefore, in what follows we will demonstrate the dissipativity of the operator $L$.  
\begin{lemma}\label{L}
Let $E^*$ be a strictly convex Banach space. The linear differential operator $L$ is dissipative.
\end{lemma}
\begin{proof}
\par Take $x\in D_L$. It is well-known that strict convexity is a geometrical property which lift from the codomain $E$ to the Bochner space $L^p(I,E)$ when $1<p<\infty$. Therefore, the duality map $J\colon L^2(I,E)\to L^2(I,E^*)$ is univalent (\cite[Prop.12.3.]{deimling}). In view of the Riesz representation theorem the duality pairing $\langle x,J(x)\rangle_{L^2}$ in $L^2(I,E)\times L^2(I,E^*)$ is given by \[\langle x,J(x)\rangle_{L^2}=\int_0^1\langle x(t),J(x)(t)\rangle\,dt.\] At the same time, $\langle x,J(x)\rangle_{L^2}=||x||_2^2$. Thus, $J(x)(t)=J(x(t))$, where the duality map on the right is the ordinary one, between the spaces $E$ and $E^*$. Let us estimate the value of the semi-inner product $\langle Lx,x\rangle_+$:
\begin{align*}
\langle Lx,x\rangle_+&=\langle x''+a_1x'+a_0x,x\rangle_+\<\langle x'',x\rangle_++\langle a_1x',x\rangle_++\langle a_0x,x\rangle_+\\&=\langle x'',J(x)\rangle_{L^2}+\langle a_1x',J(x)\rangle_{L^2}+\langle a_0x,J(x)\rangle_{L^2}\\&=\int_0^1\langle x''(t),J(x)(t)\rangle\,dt+\int_0^1\langle a_1(t)x'(t),J(x)(t)\rangle\,dt+\int_0^1\langle a_0(t)x(t),J(x)(t)\rangle\,dt\\&=\int_0^1\langle x''(t),J(x(t))\rangle\,dt+\int_0^1\langle a_1(t)x'(t),J(x(t))\rangle\,dt+\int_0^1\langle a_0(t)x(t),J(x(t))\rangle\,dt
\end{align*}
Recall that the norm $|\cdot|$ in $E$ is Gateaux-differentiable on $E\setminus\{0\}$ iff $E^*$ is strictly convex (\cite[Prop.12.2.]{deimling}). It is not difficult to show that if $f\colon I\to E$ is differentiable at $t_0\in(0,1)$ while the operator $T\colon E\to\R{}$ is Lipschitz and Gateaux-differentiable at $f(t_0)$, then the composition $T\circ f$ is differentiable at the point $t_0$, namely $(T\circ f)'(t_0)=\langle DT(f(t_0)),f'(t_0)\rangle$, where $\langle DT(x_0),v\rangle$ denotes the value of the directional derivative of $T$ at $x_0$ in the direction $v$. In respect of this, the map $(0,1)\ni t\mapsto|x(t)|\in\R{}$ is differentiable. We have the following formula for the derivative of this function: \[|x(t)|\frac{d}{dt}|x(t)|=\langle\dot{x}(t),J(x(t))\rangle.\] A fully analogous reasoning to that in \cite[Lem.4.1]{barbu} leads also to the following conclusion \[|\dot{x}(t)|\frac{d}{dt}|x(t)|=\langle\dot{x}(t),J(\dot{x}(t))\rangle.\] Continuing,
\begin{align*}
\int_0^1\langle x''(t),J(x(t))\rangle\,dt&=\int_0^1|x(t)|\frac{d}{dt}|\dot{x}(t)|\,dt=|x(t)||\dot{x}(t)|\,\rule[-5pt]{0.5pt}{15pt}_{\,0}^1-\int_0^1|\dot{x}(t)|\frac{d}{dt}|x(t)|\,dt\\&=-\int_0^1\langle\dot{x}(t),J(\dot{x}(t))\rangle\,dt=-\int_0^1|\dot{x}(t)|^2\,dt\<0.
\end{align*}
At the same time, for a certain number $\xi\in I$ we have 
\begin{align*}
\int_0^1\langle a_1(t)\dot{x}(t),J(x(t))\rangle\,dt&=\int_0^1a_1(t)\langle\dot{x}(t),J(x(t))\rangle\,dt=a_1(\xi)\int_0^1\langle\dot{x}(t),J(x(t))\rangle\,dt\\&=a_1(\xi)\f(|x(t)|^2\,\rule[-5pt]{0.5pt}{15pt}_{\,0}^1-\int_0^1|x(t)|\frac{d}{dt}|x(t)|\,dt\g)\\&=-a_1(\xi)\int_0^1\langle\dot{x}(t),J(x(t))\rangle\,dt=-\int_0^1\langle a_1(t)\dot{x}(t),J(x(t))\rangle\,dt,
\end{align*}
that is to say \[\int_0^1\langle a_1(t)\dot{x}(t),J(x(t))\rangle\,dt=0.\] Finally, \[\int_0^1\langle a_0(t)x(t),J(x(t))\rangle\,dt=\int_0^1a_0(t)\langle x(t),J(x(t))\rangle\,dt=a_0(\xi)\int_0^1|x(t)|^2\,dt\<0,\] in view of the mean value theorem. Summing up, $\langle Lx,x\rangle_+\<0$ for every $x\in D_L$ and so $L$ is dissipative.
\end{proof}
\indent The main result of this section, concerning the geometric structure of the solution set of the periodic problem \eqref{inclusion}, contains the
following:
\begin{theorem}\label{acyc}
Let $E$ be a reflexive Banach space. Suppose that $F\colon I\times E\map E$ satisfies assumptions $(\F_1)$-$(\F_6)$. Assume further that \eqref{nierzsup} holds, together with
\begin{equation}\label{bound2}
(m+1)\sup_{t\in I}||G(t,\cdot)||_2<1.
\end{equation}
Then the solution set $S_{\!F}$ of the problem \eqref{inclusion} is nonempty compact acyclic in $C^1(I,E)$.
\end{theorem}
\begin{proof}
Bearing in mind the Asplund-Trojanski theorem, we are allowed to change $|\cdot|$ to an equivalent norm such that $E$ and $E^*$ with the corresponding dual norm are locally uniformly convex.\par We will approximate the right-hand side $F$ according to the following scheme: let $F_n\colon I\times E\map E$ be such that $F_n(t,x):=F(t,x)+\frac{1}{n}x$. It is obvious that $F_n$ satisfies $(\F_1)$-$(\F_6)$. In particular, for every bounded $\Omega\subset E$ we have $\beta\f(F_n(t,\Omega)\g)\<\f(\eta(t)+\frac{1}{n}\g)\beta(\Omega)$ a.e. on $I$.\par Fix $h\in C^1(I,E)$ such that $||h||_{C^1}\<1$. Denote by $S_n^h$ the set of solutions to the following integral inclusion
\begin{equation}\label{inclusion3}
x(t)\in h(t)+\int_0^1G(t,s)F_n(s,x(s))\,ds,\;\;t\in I.
\end{equation}
Under assumption \eqref{bound2} the supremum norm of each solution $x\in S_n^h$ may be estimated as follows
\[||x||\<\frac{||h||+\sup_{t\in I}||G(t,\cdot)||_2||c||_2}{1-\f(m+\frac{1}{n}\g)\sup_{t\in I}||G(t,\cdot)||_2}\<\frac{1+\sup_{t\in I}||G(t,\cdot)||_2||c||_2}{1-\f(m+1\g)\sup_{t\in I}||G(t,\cdot)||_2}=:R_1.\] On the other hand 
\begin{align*}
||x'||&\<||h'||+\sup_{t\in I}\f\Arrowvert\frac{\partial}{\partial t} G(t,\cdot)\g\Arrowvert_2\f(||c||_2+\f(m+\frac{1}{n}\g)||x||_2\g)\\&\<||h'||+\sup_{t\in I}\f\Arrowvert\frac{\partial}{\partial t} G(t,\cdot)\g\Arrowvert_2\f(||c||_2+(m+1)R_1\g).
\end{align*}
In other words, all solutions of the inclusion \eqref{inclusion} and of the perturbated problems \eqref{inclusion3} are uniformly a priori bounded in the space $C^1(I,E)$, i.e. there exists a constant $R>0$ such that \[S_{\!F}\cup\bigcup_{||h||_{C^1}\<1}\bigcup_{n=1}^\infty S_n^h\subset D_{C^1}(0,R).\] \par Denote by $\hat{H}_n\colon C(I,\R{})\to C(I,\R{})$ the following linear operator \[\hat{H}_n(u):=2\int_0^1|G(\cdot,s)|\f(\eta(s)+n^{-1}\g)u(s)\,ds.\] Under assumption \eqref{nierzsup}, starting from a certain $n_0\in\varmathbb{N}$ we have 
\begin{equation}\label{ineq}
r\f(\hat{H}_n\g)\<||\hat{H}_n||_{\mathcal L}\<2\sup\limits_{t\in I}||G(t,\cdot)||_2\f\Arrowvert\eta+\frac{1}{n}\g\Arrowvert_2\<2\sup\limits_{t\in I}||G(t,\cdot)||_2\f(||\eta||_2+\frac{1}{n}\g)<1
\end{equation}
for every $n\geqslant n_0$. On the basis of the findings made in the proof of Theorem \ref{solset}. we infer that the set-valued operator $\Phi_n\colon D_{C^1}(0,R)\map C^1(I,E)$, defined by $\Phi_n:=H\circ N_{F_n}$, is condensing relative to MNC $\psi_{C^1}$ for $n$ large enough. It is a well-established knowledge that a multivalued vector field associated with a compact valued usc map condensing relative to some monotone, algebraically semiadditive and regular MNC, is proper. Therefore, one can associate to maps $\Phi_n$ and $\Phi$ convex compact valued upper semicontinuous and proper vector fields $\Psi_n:=id-\Phi_n$ and $\Psi:=id-\Phi$ (where $id$ denotes the identity).\par Take $x\in D_{C^1}(0,R)$. If $y\in\Psi_n(x)$, then there is $w\in N_{F_n}(x)$ such that $y=x-H(w)$. Put $z:=x-H\f(w-\frac{1}{n}x\g)$. Then 
\begin{align*}
||y-z||_{C^1}&=\f\Arrowvert x-H(w)-\f(x-H\f(w-\frac{1}{n}x\g)\g)\g\Arrowvert_{C^1}=\f\Arrowvert H\f(\frac{1}{n}x\g)\g\Arrowvert_{C^1}\<\frac{||H||_{\mathcal L}||x||_2}{n}\\&\<\frac{||H||_{\mathcal L}||x||_{C^1}}{n}\<\frac{||H||_{\mathcal L}R}{n},
\end{align*} 
whence it follows that there is $(\eps_n)_\n\subset\R{}$ such that $\eps_n\to 0^+$ and $\Psi_n(x)\subset B_{C^1}(\Psi(x),\eps_n)$ for every $x\in D_{C^1}(0,R)$. In fact, $d_{C^1}(\Psi_n(x),\Psi(x))\to 0$ as $n\to\infty$ (where $d_{C^1}$ is the Hausdorff distance).\par In order to verify (ii) in Theorem \ref{gabor}., assume that $0\in\Psi(x)$. It means that $x=H(w)$ and $x'=\frac{d}{dt}H(w)$ for some $w\in N_F(x)$. In other words, $||x-H(w)||_{C^1}=0$. Put $u:=w+\frac{1}{n}x$ and $y:=x-H(u)$. Obviously, $u\in N_{F_n}(x)$ and $y\in\Psi_n(x)$. A straightforward calculation proves that 
\[||y||_{C^1}=\f\Arrowvert x-H\f(w+\frac{1}{n}x\g)\g\Arrowvert_{C^1}\<||x-H(w)||_{C^1}+\f\Arrowvert H\f(\frac{1}{n}x\g)\g\Arrowvert_{C^1}=\frac{||H(x)||_{C^1}}{n}\<\frac{||H||_{\mathcal L}R}{n}.\] Hence, $y\in\Psi_n(x)\cap D_{C^1}(0,\eps_n)$.\par One can readily check that the Nemytski\v{\i} operator \[N_{F_n}\f(C^1(I,E),||\cdot||_2\g)\map\f(L^2(I,E),||\cdot||_2\g),\] corresponding to $F_n$, is strongly accretive, and more specifically
\begin{equation}\label{strdis}
\langle N_{F_n}(x)-N_{F_n}(y),x-y\rangle_+\geqslant \frac{1}{n}||x-y||_2^2
\end{equation} 
for all $x,y\in C^1(I,E)$. Let us also note that \[||x_1-x_2||_{C^1}>0\Rightarrow ||x_1-x_2||_2>0.\] Indeed, if $x_1,x_2\in C^1(I,E)$ and $x_1(t)=x_2(t)$ a.e. in $I$, then \[\int_0^1\dot{x}_1(t)\varphi(t)\,dt=-\int_0^1x_1(t)\varphi'(t)\,dt=-\int_0^1x_2(t)\varphi'(t)\,dt=\int_0^1\dot{x}_2(t)\varphi(t)\,dt,\] whence \[\int_0^1(\dot{x}_1(t)-\dot{x}_2(t))\varphi(t)\,dt=0\] for all $\varphi\in C^\infty_c((0,1))$. By virtue of du Bois-Reymond's lemma (\cite[Prop.2.2.19]{papa}) we have $\dot{x}_1(t)=\dot{x}_2(t)$ a.e. in $I$. Thus, $x_1(t)=x_2(t)$ and $\dot{x}_1(t)=\dot{x}_2(t)$ for each $t\in I$. Consequently, $||x_1-x_2||_{C^1}=0$.\par Fix $h\in B_{C^1}(0,\eps_n)$ for sufficiently large $n$. Everyone should easily realize that the set of solutions of the inclusion $h\in\Psi_n(x)$ coincides with the solution set $S_n^h$ of the Hammerstein inclusion \eqref{inclusion3}. Therefore, in order to take advantage of Theorem \ref{gabor}. it is sufficient to show that $S_n^h$ is nonempty acyclic. Taking into account Remarks \ref{rem1}. and \ref{rem2}. concerning Theorem \ref{solset}. and having regard that inequality \eqref{ineq} is valid for $n$ large enough, the non-emptiness of the solution set $S_n^h$ can be considered justified.\par We claim that inclusion \eqref{inclusion3} possesses at most one solution. Suppose to the contrary that there are two different solutions $x_1, x_2\in C^1(I,E)$ of the problem \eqref{inclusion3}, which means that there exist also two Bochner integrable selections $w_1\in N_{F_n}(x_1)$ and $w_2\in N_{F_n}(x_2)$ such that $x_1=h+H(w_1)$ and $x_2=h+H(w_2)$. Consequently, $L(x_1-x_2)=w_1-w_2$. As we established previously, $||x_1-x_2||_2>0$. Now, applying Lemma \ref{L}. and property \eqref{strdis} we deduce that 
\begin{align*}
0&=\langle -L(x_1-x_2)+(w_1-w_2),x_1-x_2\rangle_+\geqslant\langle -L(x_1-x_2),x_1-x_2\rangle_++\langle w_1-w_2,x_1-x_2\rangle_+\\&\geqslant\langle w_1-w_2,x_1-x_2\rangle_+\geqslant\frac{1}{n}||x_1-x_2||_2^2>0
\end{align*}
- a contradiction. Actually, we have shown that the solution set $S_n^h$ is a singleton. In view of Theorem \ref{gabor}. the set $\f\{x\in D_{C^1}(0,R)\colon 0\in\Psi(x)\g\}$ is compact acyclic. In other words, the solution set $S_{\!F}$ of the original problem \eqref{inclusion} is compact acyclic as well.
\end{proof}
\begin{corollary}
Under assumptions of Theorem \ref{acyc}. the solution set $S_{\!F}$ of the periodic boundary value problem \eqref{inclusion} forms a nonempty compact acyclic subset of the space $\Hi^2(0,1;E)$ endowed with the weak topology $\sigma\f(\Hi^2,\f(\Hi^2\g)^*\g)$.
\end{corollary}
\begin{proof}
Take $u_n\in S_{\!F}$ such that $||u_n-u||_{C^1}\to 0$ as $n\to\infty$. Since $u_n$ is a solution of \eqref{inclusion} we have $Lu_n=f_n$ for some $f_n\in N_F(u_n)$. Recall that the Nemytski\v{\i} operator $N_F$ is weakly upper semicontinuous. Hence, passing to a subsequence if necessary, we may assume that $(f_n)_\n$ converges weakly in $L^2(I,E)$ to a particular $f\in N_F(u)$. Observe that $u_n=H(f_n)\rightharpoonup H(f)$ in $\Hi^2(0,1;E)$, because the Hammerstein operator $H\colon L^2\to\Hi^2$ is linear continuous. Of course, $H(f_n)\rightharpoonup H(f)$ in $\Hi^1(0,1;E)$ as well. On the other hand, we know that $u_n\to u$ in $\Hi^1(0,1;E)$. Consequently, $(u_n)_\n$ converges weakly to $u$ in $\Hi^2(0,1;E)$. \par We have shown that the identity operator $id\colon \f(S_{\!F},||\cdot||_{C^1}\g)\to\Hi^2(0,1;E)$ is demicontinuous. By Theorem \ref{solset}. the solution set $S_{\!F}$ is compact in $\f(\Hi^2(0,1;E),\sigma\f(\Hi^2,\f(\Hi^2\g)^*\g)\g)$. Thus, $id\colon\f(S_{\!F},||\cdot||_{C^1}\g)\to\f(S_{\!F},\sigma\f(\Hi^2,\f(\Hi^2\g)^*\g)\g)$ is a continuous mapping between compact topological spaces. Consequently, the graded vector space $H_*\f(\f(S_{\!F},\sigma\f(\Hi^2,\f(\Hi^2\g)^*\g)\g)\g)$ of \v{C}ech homologies with coefficients in $\varmathbb{Q}$ is isomorphic to the graded vector space $H_*\f(\f(S_{\!F},||\cdot||_{C^1}\g)\g)$. By virtue of Theorem \ref{acyc}. the reduced homologies $\w{H}_*\f(\f(S_{\!F},||\cdot||_{C^1}\g)\g)$ are trivial. Therefore, the solution set $S_{\!F}$ is acyclic as a subset of the space $\f(\Hi^2(0,1;E),\sigma\f(\Hi^2,\f(\Hi^2\g)^*\g)\g)$.
\end{proof}
\begin{corollary}
Under assumptions of Theorem \ref{acyc}. the solution set $S_{\!F}$ of the periodic boundary value problem \eqref{inclusion} forms a continuum in the space $C^1(I,E)$. The solution set $S_{\!F}$ is also a continuum as a subspace of $\f(\Hi^2(0,1;E),\sigma\f(\Hi^2,\f(\Hi^2\g)^*\g)\g)$.
\end{corollary}
\begin{corollary}
Suppose that $F\colon I\times\R{N}\map\R{N}$ satisfies assumptions $(\F_1)$-$(\F_4)$ and the multimap $F(t,\cdot)$ is monotone for a.a. $t\in I$. Assume further that \eqref{bound2} holds. Then the solution set $S_{\!F}$ of the periodic boundary value problem \eqref{inclusion} is nonempty compact acyclic in the space $C^1\f(I,\R{N}\g)$ as well as in the space $\Hi^2\f(0,1;\R{N}\g)$ endowed with the weak topology.
\end{corollary}
\begin{corollary}
Let $E$ be a reflexive Banach space. Suppose that the univalent mapping $f\colon I\times E\to E$ satisfies
\begin{itemize}
\item[$(f_1)$] the map $f(\cdot,x)$ is strongly measurable for every $x\in E$,
\item[$(f_2)$] the map $f(t,\cdot)$ is demicontinuous for a.a. $t\in I$, i.e. $f(t,\cdot)\colon(E,|\cdot|)\to(E,\sigma(E,E^*))$ is continuous,
\item[$(f_3)$] there is $c\in L^2(I,\R{})$ and $m\geqslant 0$ such that $|f(t,x)|\<c(t)+m|x|$ for a.a. $t\in I$ and for all $x\in E$,
\item[$(f_4)$] there is a function $\eta\in L^2(I,\R{})$ such that for all bounded subsets $\Omega\subset E$ and for a.a. $t\in I$ the inequality holds \[\beta(f(t,\Omega);E)\<\eta(t)\beta(\Omega;E),\]
\item[$(f_5)$] the map $f(t,\cdot)$ is accretive for a.a. $t\in I$.
\end{itemize}
Assume further that \eqref{nierzsup} and \eqref{bound2} holds. Then the solution set $S_{\!f}$ of the periodic boundary value problem \eqref{inclusion}, with the right-hand side $F$ replaced by the mapping $f$, is a compact $R_\delta$ in the space $C^1(I,E)$.
\end{corollary}
\begin{proof}
Observe that the single-valued Nemytski\v{\i} operator $N_f\colon C^1(I,E)\to L^2(I,E)$, corresponding to the mapping $f$, is demicontinuous. It is clear in view of \cite[Prop.2.]{pietkun}, given that the dual $E^*$ is reflexive.\par Let $f_n\colon I\times E\to E$ be such that $f_n(t,x):=f(t,x)+\frac{1}{n}x$. The maps $\Phi:=H\circ N_f$ and $\Phi_n:=H\circ N_{f_n}$, introduced previously in the proof of Theorem \ref{acyc}., are now continuous univalent functions, condensing relative to MNC $\psi_{C^1}$. Consequently, the vector fields $\Psi,\Psi_n\colon D_{C^1}(0,R)\to C^1(I,E)$ given by $\Psi:=id-\Phi$, $\Psi_n:=id-\Phi_n$, are proper continuous. In order to prove the thesis it is sufficient to verify that there exists a sequence of positive numbers $(\eps_n)_\n$ converging to zero such that the following two conditions hold:
\begin{itemize}
\item[(a)] $||\Psi_n(x)-\Psi(x)||_{C^1}\<\eps_n$ for every $x\in D_{C^1}(0,R)$,
\item[(b)] for every $\n$ and every $h\in C^1(I,E)$ with $||h||_{C^1}\<\eps_n$ the equation $\Psi_n(x)=h$ has a unique solution.
\end{itemize}
\par Take $x\in D_{C^1}(0,R)$. An easy calculation shows that 
\[||\Psi_n(x)-\Psi(x)||_{C^1}\!=||H(N_{f_n}(x)-N_f(x))||_{C^1}\!\<||H||_{\mathcal L}||N_{f_n}(x)-N_f(x)||_2\!=\!\frac{||H||_{\mathcal L}||x||_2}{n}\<\!\frac{||H||_{\mathcal L}R}{n},\]
whence (a) already follows. The proof of the property (b) does not deviate in any way from the justification that the integral inclusion \eqref{inclusion3} possesses a unique solution. The application of the well-known Browder-Gupta theorem completes the proof.
\end{proof}
\indent The issue concerning description of the topology of solutions to periodic problems for ordinary differential equations of second order is already present in the literature on the subject. The following theorem extends the thesis of \cite[Th.6.]{papa3} beyond the scalar case. Let us mention that the authors of \cite{papa3} exploited in their work the method of lower and upper solutions instead of a coercivity condition (ii).
\begin{theorem}\label{convex}
Assume that the Carath\'eodory function $f\colon I\times\R{N}\to\R{N}$ satisfies
\begin{itemize}
\item[(i)] there is $c\in L^2(I,\R{})$ and $m>0$ such that $|f(t,x)|\<c(t)+m|x|$ for a.a. $t\in I$ and for all $x\in\R{N}$,
\item[(ii)] $\langle x,f(t,x)\rangle\geqslant c_1|x|^2-a(t)\,$ for a.a. $t\in I$ and all $x\in\R{N}$ with $a\in L^1(I,\R{}_+)$ and $c_1>0$,
\item[(iii)] the map $f(t,\cdot)$ is monotone for $t\in I$.
\end{itemize}
Then the solution set $S_{\!f}$ of the following periodic boundary value problem
\begin{equation}\label{papa}
\begin{gathered}
x''(t)=f(t,x(t))-Mx'(t)\;\;\;\text{a.e. on } I,\\
x(0)=x(1),\\
x'(0)=x'(1),
\end{gathered}
\end{equation}
where $M>0$, is nonempty and convex. Moreover, if $c_1>m$, then $S_{\!\!f}$ is also weakly compact in $\Hi^2\f(0,1;\R{N}\g)$.
\end{theorem}
\begin{proof}
Let us introduce the following notation: \[D_L:=\left\{x\in\Hi^2(0,1;\R{N})\colon x(0)=x(1)\text{ and }x'(0)=x'(1)\right\}\] and $L\colon D_L\subset L^2(I,\R{N})\to L^2(I,\R{N})$ is a differential operator such that $L:=\frac{d^2}{dt^2}+M\frac{d}{dt}$. By $N_f\colon L^2(I,\R{N})\to L^2(I,\R{N})$ we denote the Nemytski\v{\i} operator corresponding to $f$, i.e. $N_f(x):=f(\cdot,x(\cdot))$. In view of Krasnoselskii's theorem (\cite[Th.3.4.4]{papa}), the mapping $N_f$ is well-defined, continuous and bounded on bounded sets. Observe that the solution set $S_{\!f}$ of \eqref{papa} is nothing more than the preimage $((-L)+N_f)^{-1}(\{0\})$.\par In view of Lemma \ref{L}, the operator $-L$ is monotone. It is a matter of easy calculations to show that the reduced system
\begin{equation}
\begin{gathered}
x''+Mx'-x=0\;\;\;\text{on } I,\\
x(0)=x(1),\\
x'(0)=x'(1)
\end{gathered}
\end{equation}
is incompatible. This means that for every $w\in L^2(I,\R{N})$ there exists a (unique) solution $x\in D_L$ of the problem $-Lx+x=w$. Consequently, $(-L+id)\colon D_L\to L^2(I,\R{N})$ is onto. Therefore, operator $-L$ is maximal monotone (\cite[Prop.4.3.]{deimling2}). On the other hand, the Nemytski\v{\i} operator $N_f$ is also maximal monotone (\cite[Prop.3.4.6]{papa}). Now, the idea here is to note that the operator $(-L+N_f)\colon D_L\cap L^2(I,\R{N})\to L^2(I,\R{N})$ is maximal monotone as well. We may apply \cite[Th.11.4(a)]{deimling} to justify this property.\par Let us find out that the operator $-L+N_f$ is coercive, i.e. \[\frac{\langle(-L+N_f)(x),x\rangle_{L^2}}{||x||_2}\xrightarrow[||x||_2\to +\infty]{}+\infty.\]
We have 
\begin{align*}
\frac{\langle(-L+N_f)(x),x\rangle_{L^2}}{||x||_2}&=\frac{\langle -Lx,x\rangle_{L^2}+\langle N_f(x),x\rangle_{L^2}}{||x||_2}\geqslant\frac{\langle N_f(x),x\rangle_{L^2}}{||x||_2}\\&=\frac{1}{||x||_2}\int_0^1\langle f(t,x(t)),x(t)\rangle\,dt\geqslant\frac{1}{||x||_2}\int_0^1c_1|x(t)|^2-a(t)\,dt\\&=\frac{c_1||x||_2^2-||a||_1}{||x||_2}\xrightarrow[||x||_2\to +\infty]{}+\infty.
\end{align*}
In particular, $-L+N_f$ is weakly coercive (in the sense of \cite[Def.3.2.1]{papa}) and we are entitled to take advantage of \cite[Cor.3.2.31]{papa} and \cite[Prop.3.2.34]{papa}. As a result the preimage $((-L)+N_f)^{-1}(\{0\})$ is a nonempty and convex set. It is also a bounded subset of the Lebesgue space $L^2(I,\R{N})$.\par We may assume w.l.o.g that $|f(0,x)|\<c(0)+m|x|<+\infty$ and $\langle x,f(0,x)\rangle\geqslant c_1|x|^2$ for all $x\in\R{N}$. If $x\in S_{\!f}$, then $c_1|\dot{x}(0)|^2\<\langle \dot{x}(0),f(0,\dot{x}(0))\rangle\<|\dot{x}(0)|(c(0)+m|\dot{x}(0)|)$, i.e. $|\dot{x}(0)|\<\frac{c(0)}{c_1-m}$. The periodic problem \eqref{papa} can be easily put into equivalent Sturm-Liouville form
\begin{equation}
\begin{gathered}
\f(e^{Mt}x'(t)\g)'=e^{Mt}f(t,x(t))\;\;\;\text{a.e. on } I,\\
x(0)=x(1),\\
x'(0)=x'(1).
\end{gathered}
\end{equation}
Let $R>0$ be such that $S_{\!f}\subset D_{L^2}(0,R)\subset L^2(I,\R{N})$. Then \[\f\Arrowvert\f(e^{M(\cdot)}x'\g)'\g\Arrowvert_2\!=\!\f\Arrowvert e^{M(\cdot)}N_f(x)\g\Arrowvert_2\<e^M||N_f(x)||_2\<e^M\f(||c||_2+m||x||_2\g)\<e^M\f(||c||_2+mR\g):=R_1.\] On the other hand, we see that 
\begin{align*}
\f\Arrowvert e^{M(\cdot)}x'\g\Arrowvert_2&=\f\Arrowvert\dot{x}(0)+\int_0^{(\cdot)}\f(e^{Ms}x'(s)\g)'\,ds\g\Arrowvert_2\<|\dot{x}(0)|+\f(\int_0^1\f|\int_0^t\f(e^{Ms}x'(s)\g)'\,ds\g|^2\,dt\g)^{\frac{1}{2}}\\&\<|\dot{x}(0)|+\f(\int_0^1\f(\int_0^1\f|\f(e^{Ms}x'(s)\g)'\g|\,ds\g)^2\,dt\g)^{\frac{1}{2}}\<|\dot{x}(0)|+\f(\int_0^1\f\Arrowvert\f(e^{M(\cdot)}x'\g)'\g\Arrowvert_2^2\,dt\g)^{\frac{1}{2}}\\&=|\dot{x}(0)|+\f\Arrowvert\f(e^{M(\cdot)}x'\g)'\g\Arrowvert_2\<\frac{c(0)}{c_1-m}+R_1.
\end{align*}
The mean value theorem indicates that there is $\xi\in I$ such that 
\[||x'||_2\<e^{-M\xi}\f(\frac{c(0)}{c_1-m}+R_1\g)\<\frac{c(0)}{c_1-m}+R_1=:R_2.\] Now, the estimation of the $L^2$-norm of the second derivative follows immediately \[||x''||_2=||N_f(x)-Mx'||_2\<||N_f(x)||_2+M||x'||_2\<||c||_2+mR+MR_2=:R_3.\] Eventually, $||x||_{\Hi^2}=||x||_2+||x'||_2+||x''||_2\<R+R_2+R_3$ for all $x\in S_{\!\!f}$. Clearly, the solution set $S_{\!\!f}$ is relatively weakly compact as a bounded subset of the reflexive Sobolev space $\Hi^2(0,1;\R{N})$.\par Let $(x_n)_\n\subset S_{\!\!f}$ and assume that $x_n\to x$ in $\Hi^2(0,1;\R{N})$. Observe that $x\in D_L$, since $D_L$ forms a closed (linear) subspace of $\Hi^2(0,1;\R{N})$. Knowing that operators $L$ and $N_f$ are continuous we deduce $Lx=\lim\limits_{n\to\infty}Lx_n=\lim\limits_{n\to\infty}N_f(x_n)=N_f(x)$. Therefore, $x\in S_{\!\!f}$ and the solution set $S_{\!\!f}$ is closed in $\Hi^2(0,1;\R{N})$. Since $S_{\!\!f}$ is also convex, it must be weakly compact in $\Hi^2(0,1;\R{N})$.
\end{proof}

\section{the Lipschitz case}
Let $B_1$ and $B_2$ be the boundary conditions operators corresponding to \eqref{inclus}, i.e. 
\begin{equation*}
\begin{gathered}
B_1x:=b_{11}x(0)+b_{12}x'(0)+c_{11}x(1)+c_{12}x'(1),\\
B_2x:=b_{21}x(0)+b_{22}x'(0)+c_{21}x(1)+c_{22}x'(1).
\end{gathered}
\end{equation*}
\indent The following assumption is our standing hypothesis for the rest of this section. \par\vspace{0.5\baselineskip}\noindent
{\bf Assumption $\mathbf{(G_2)}$:}
\begin{minipage}[t]{9.5cm}
The coefficient mappings $a_2,a_1,a_0\colon I\to\R{}$ are continuous. The reduced system (that is the scalar completely homogeneous boundary value problem )
\begin{equation}\label{equation2}
\begin{gathered}
a_2(t)x''+a_1(t)x'+a_0(t)x=0\;\;\;\text{on }I,\\
B_1x=0,\\
B_2x=0
\end{gathered}
\end{equation}
is incompatible, i.e. possesses only the trivial solution. 
\end{minipage}
\par\vspace{0.5\baselineskip}
In what follows we will make use of a relaxed notion of the solution to problem \eqref{inclus}, namely a function $x\colon(0,1)\to E$ is said to be a solution to \eqref{inclus} if $x$ is an element of $W^{2,1}(0,1;E)$ such that the boundary conditions $B_1x=d_1$ and $B_2x=d_2$ are met and the differential inclusion $a_2(t)x''(t)+a_1(t)x'(t)+a_0(t)x(t)\in F(t,x(t))$ is satisfied for a.a. $t\in I$.\par 
Under assumption $\mathbf{(G_2)}$ there exists a unique solution $h\in C^2(I,E)$ of the problem 
\begin{equation*}
\begin{gathered}
a_2(t)h''+a_1(t)h'+a_0(t)h=0\;\;\;\text{on }I,\\
B_1h=d_1,\\
B_2h=d_2
\end{gathered}
\end{equation*}
and uniquely designated Green's function $G\in C(I\times I,\R{})$ for the BVP \eqref{equation2}. Hence, in the context of the above definition, problem \eqref{inclus} is equivalent to the fixed point problem 
\begin{equation}\label{inclusion4}
x\in h+\int_0^1G(\cdot,s)F(s,x(s))\,ds=h+H\circ N_F(x)
\end{equation}
on $W^{2,1}(0,1;E)$, where the operators of Nemytski\v{\i} and Hammerstein possess adequately "enlarged" domain and range, i.e. $N_F\colon W^{2,1}(0,1;E)\map L^1(I,E)$ and $H\colon L^1(I,E)\to W^{2,1}(0,1;E)$.\par To study problem \eqref{inclus} we introduce the following assumptions about the set-valued map $F\colon I\times E\map E$.
\begin{itemize}
\item[$(\h_1)$] the set $F(t,x)$ is nonempty closed and bounded for every $(t,x)\in I\times E$, 
\item[$(\h_2)$] the map $F(\cdot,x)$ is ${\mathcal L}(I)$-measurable for every $x\in E$,
\item[$(\h_3)$] $F$ is Lipschitz with respect to the second argument, i.e. there exists $\mu\in L^1(I,\R{})$ such that $d_E(F(t,x),F(t,y))\<\mu(t)|x-y|$, for all $x,y\in E$, a.e. in $I$,
\item[$(\h_4)$] there exists $\alpha\in L^1(I,\R{})$ such that $d_E(0,F(t,0))\<\alpha(t)$, a.e. in $I$.
\end{itemize}
\par If the right-hand side $F$ is Lipschitz with respect to $x$, then we can say much more about the geometry of the solution set $S_{\!F}$ to BVP \eqref{inclus} than it is merely acyclic. As it is shown in the following result $S_{\!F}$ is even an absolute retract.
\begin{theorem}\label{structure}
Let $E$ be a separable Banach space. Assume that $F\colon I\times E\map E$ satisfies $(\h_1)$-$(\h_4)$. Assume further that 
\begin{equation}\label{norm}
||\mu||_1<\f(\sup_{(t,s)\in I^2}|G(t,s)|\g)^{-1}\!.
\end{equation}
Then the solution set $S_{\!F}$ of the two-point boundary value problem \eqref{inclus} is a retract of the space $W^{2,1}(0,1;E)$.
\end{theorem}
\begin{proof}
Since 
\begin{align*}
x\in\fix(h+H\circ N_F)&\Leftrightarrow\underset{w}{\exists}\; w\in N_F(x)\wedge x=h+H(w)\\&\Leftrightarrow\underset{w}{\exists}\; w\in N_F(h+H(w))\wedge x=h+H(w)\\&\Leftrightarrow\underset{w}{\exists}\; w\in\fix(N_F(h+H(\cdot)))\wedge x=h+H(w)\\&\Leftrightarrow x\in h+H(\fix(N_F(h+H(\cdot)))),
\end{align*}
we see that there is a coincidence $S_{\!F}=h+H(\fix(N_F(h+H(\cdot))))$. We claim that the fixed point set $\fix(N_F(h+H(\cdot)))$ is a retract of the space $L^1(I,E)$.\par First of all, let us evaluate the norm of the bounded linear Hammerstein integral operator $H\colon L^1(I,E)\to C(I,E)$,
\begin{equation}\label{normofH}
\begin{split}
||H||_{\mathcal L}&=\sup_{||w||_1=1}||H(w)||\<\sup_{||w||_1=1}\sup_{t\in I}\int_0^1|G(t,s)||w(s)|ds\<\sup_{||w||_1=1}\sup_{t\in I}||G(t,\cdot)||\,||w||_1\\&=\sup_{(t,s)\in I^2}|G(t,s)|.
\end{split}
\end{equation}
\par It follows from $(\h_3)$ that $F(t,\cdot)$ is Hausdorff continuous. By virtue of \cite[Th.3.3]{papa2} the map $F$ is ${\mathcal L}(I)\otimes{\mathcal B}(E)$-measurable. In particular $F(\cdot,h(\cdot)+H(u)(\cdot))$ is measurable for $u\in L^1(I,E)$, since $F$ is superpositionally measurable (\cite[Th.1.]{zyg}). Thanks to $(\h_3)$ and $(\h_4)$ we have that the function $I\ni t\mapsto\inf\{|z|\colon z\in F(t,h(t)+H(u)(t))\}$ is integrable in the sense of Lebesgue, which means that $N_F(h+H(u))\neq\varnothing$ for any $u\in L^1(I,E)$.\par Fix arguments $u_1,u_2\in L^1(I,E)$. Take an arbitrary $w_1\in N_F(h+H(u_1))$ and $\eps>0$. It is clear that $t\mapsto d(w_1(t),F(t,h(t)+H(u_2)(t)))$ is measurable. By $(\h_3)$ we have \[d(w_1(t),F(t,h(t)+H(u_2)(t)))\<\mu(t)|H(u_1-u_2)(t)|\;\;\text{a.e. on }I,\] where $\mu(\cdot)|H(u_1-u_2)(\cdot)|\in L^1(I,\R{})$. Thus, $d(w_1(\cdot),F(\cdot,h(\cdot)+H(u_2)(\cdot)))\in L^1(I,\R{})$. Define $\phi:=\essinf\{|u(\cdot)|\colon u\in K\}$, where $K:=\{w_1\}-N_F(h+H(u_2))$ is a nonempty closed and decomposable subset of $L^1(I,E)$. Applying the Castaing representation we may write $F(t,h(t)+H(u_2)(t))=\overline{\{g_n(t)\}}_\n$. Thus, for every $\n$, $\phi(t)\<|w_1(t)-g_n(t)|$ a.e. in $I$. Obviously, there is a subset $J\subset I$ of full measure such that for every $t\in J$, $\phi(t)\<\inf_\n|w_1(t)-g_n(t)|$. Consequently, 
\begin{align*}
\phi(t)\<d\left(w_1(t),\{g_n(t)\}_\n\right)&=d\left(w_1(t),\overline{\{g_n(t)\}}_\n\right)=d(w_1(t),F(t,h(t)+H(u_2)(t)))\\&<d(w_1(t),F(t,h(t)+H(u_2)(t)))+\eps
\end{align*}
a.e. on $I$. It follows by \cite[Prop.2.]{brescol} that there is $w_2\in N_F(h+H(u_2))$ such that \[|w_1(t)-w_2(t)|<d(w_1(t),F(t,h(t)+H(u_2)(t)))+\eps\;\;\text{for a.a. }t\in I.\] Thus, we can estimate
\begin{align*}
||w_1-w_2||_1&=\int_0^1|w_1(t)-w_2(t)|\,dt\<\int_0^1(d(w_1(t),F(t,h(t)+H(u_2)(t)))+\eps)\,dt\\&\<\int_0^1d_E(F(t,h(t)+H(u_1)(t)),F(t,h(t)+H(u_2)(t))\,dt+\eps\\&\<\int_0^1\mu(t)|H(u_1)(t)-H(u_2)(t)|\,dt+\eps\<||\mu||_1||H(u_1-u_2)||+\eps\\&\<||\mu||_1||H||_{\mathcal L}||u_1-u_2||_1+\eps
\end{align*}
Since $\eps$ was arbitrarily small and $w_1$ was an arbitrary element of $N_F(h+H(u_1))$ it follows that \[\sup_{w\in N_F(h+H(u_1))}d(w,N_F(h+H(u_2)))\<||\mu||_1||H||_{\mathcal L}||u_1-u_2||_1\] and consequently \[d_{L^1}(N_F(h+H(u_1)),N_F(h+H(u_2)))\<||\mu||_1||H||_{\mathcal L}||u_1-u_2||_1\] for every $u_1,u_2\in L^1(I,E)$. In view of \eqref{norm} and \eqref{normofH} we conclude that the set-valued operator $N_F(h+H(\cdot))\colon L^1(I,E)\map L^1(I,E)$ is contractive.\par Take $w_1\in N_F(h+H(u))$. As we have seen, for each $\eps>0$ there exists $w_2\in N_F(0)$ satisfying $|w_1(t)-w_2(t)|<d(w_1(t),F(t,0))+\eps$ a.e. on $I$. Thus, 
\begin{align*}
|w_1(t)|&\<|w_1(t)-w_2(t)|+|w_2(t)|\<d_E(F(t,h(t)+H(u)(t)),F(t,0))+\eps+d_E(0,F(t,0))\\&\<\mu(t)|h(t)+H(u)(t)|+\eps+\alpha(t),
\end{align*}
by $(\h_3)$ and $(\h_4)$. In other words, operator $N_F(h+H(\cdot))$ possesses nonempty bounded values. It is a matter of routine to check that they are also closed and decomposable. Now, we can invoke \cite[Th.1.]{bres} to indicate a retraction $r\colon L^1(I,E)\to\fix(N_F(h+H(\cdot)))$.\par Denote by $L\colon W^{2,1}(0,1;E)\to L^1(I,E)$ a linear bounded differential operator corresponding to the problem \eqref{inclus}, i.e. $Lx:=a_2x''+a_1x'+a_0x$.
Let $R\colon W^{2,1}(0,1;E)\to S_{\!F}$ be a map which associates with each $x\in W^{2,1}(0,1;E)$ a function $R(x):=h+H(r(Lx))$. It is clear that the continuity of the composite function $R$ follows directly from the continuity of its components $H$, $r$ and $L$. \par It remained to us justify that $\res{R}{S_{\!F}}=id_{S_{\!F}}$. Suppose $x$ is a solution of \eqref{inclus}. Then $x=h+H(w)$ for some $w\in N_F(x)$. In fact, $Lx=Lh+LH(w)=w$ and $w\in N_F(h+H(w))$. Hence, $Lx\in\fix(N_F(h+H(\cdot)))$ and $r(Lx)=Lx$. Eventually, $R(x)=h+H(r(Lx))=h+H(Lx)=h+H(w)=x$. In conclusion, the mapping $R$ gives a retraction of the space $W^{2,1}(0,1;E)$ onto the solution set $S_{\!F}$.
\end{proof}
\begin{remark}
In view of Theorem \ref{structure}, the space of solutions $S_{\!F}$ with the $W^{2,1}(0,1;E)$ metric is an absolute retract. In particular, $S_{\!F}$ forms a nonempty and closed subset of the space $W^{2,1}(0,1;E)$.
\end{remark}
\begin{example}
Consider \eqref{inclus} with the coefficients $a_2\equiv 1$, $a_1\equiv 0$, $a_0\equiv -1$ and the boundary conditions being periodic. Then the BVP \eqref{inclus} reduces to the problem \eqref{reduced}. The influence function $G$ for the problem \eqref{reduced} is given by \eqref{green}. It is easy to calculate that \[\sup_{(t,s)\in I^2}|G(t,s)|=-\frac{1}{2}\frac{e+1}{1-e}.\] Therefore, in this particular case condition \eqref{norm} assumes the form \[||\mu||_1<-2\frac{1-e}{e+1}\approx 0.924.\]
\end{example}


\begin{thebibliography}{99}
\bibitem{sadovski} R. Akhmerov, M. Kamenskii, A. Potapov, A. Rodkina, B. Sadovskii, {\it Measures of noncompactness and condensing operators}, Birkh\"{a}user Verlag, 1992. 
\bibitem{andres} J. Andres, L. G\'orniewicz, {\it Topological Fixed Point Principles for Boundary Value Problems}, Kluwer Academic Publishers, Dordrecht, 2003. 
\bibitem{barbu} V. Barbu, {\it Nonlinear Differential Equations of Monotone Types in Banach Spaces}, Springer, New York, 2010.
\bibitem{bres} A. Bressan, A. Cellina, A. Fryszkowski, {\it A class of absolute retracts in spaces of integrable functions}, Proc. Amer. Math. Soc. 112 (1991), 413-418.
\bibitem{brescol} A. Bressan, G. Colombo, {\it Extensions and selections of maps with decomposable values}, Studia Math. 90 (1988), 69-86.
\bibitem{deimling} K. Deimling, {\it Nonlinear Functional Analysis}, Springer-Verlag, Berlin, 1985.
\bibitem{deimling2} K. Deimling, {\it Multivalued differential equations}, Walter de Gruyter, Berlin-New York, 1992.
\bibitem{djebali} S. Djebali, L. G\'orniewicz, A. Ouahab, {\it Solution Sets for Differential Equations and Inclusions}, Walter de Gruyter, Berlin-Boston, 2013.
\bibitem{eastham} M. Eastham, {\it The spectral theory of periodic differential equations}, Scottish Academic Press, 1973. 
\bibitem{gabor} G. Gabor, {\it On the acyclicity of fixed point sets of multivalued maps}, Topol. Methods Nonlinear Anal. 14 (1999), 327-343.
\bibitem{papa} L. Gasi\'nski, N. Papageorgiou, {\it Nonlinear Analysis}, Taylor \& Francis Group, Boca Raton, 2005. 
\bibitem{gorn} L. G\'orniewicz, {\it Topological fixed point theory of multivalued mappings}, Second ed., Springer, Dordrecht, 2006.
\bibitem{heinz} H. P. Heinz, {\it On the behaviour of measures of noncompactness with respect to differentiation and integration of vector-valued functions}, Nonlinear Anal. 7 (1983), 1351-1371.
\bibitem{hyman} D. M. Hyman, {\it On decreasing sequences of compact absolute retracts}, Fund. Math. 64 (1969), 91-97.
\bibitem{obu} M. Kamenskii, V. Obukhovskii, P. Zecca, {\it Condensing multivalued maps and semilinear differential inclusions in Banach spaces}, Nonlinear Anal. Appl., vol.7, Walter de Gruyter, Berlin-New York, 2001.
\bibitem{monch} H. M\"{o}nch, {\it Boundary value problems for nonlinear ordinary differential equations of second order in Banach spaces}, Nonlinear Anal. 4 (1980), 985-999.
\bibitem{ponosov} Y. Nepomnyashchikh, A. Ponosov, {\it The necessity of the Carath\'eodory conditions for the lower semicontinuity in measure of the multivalued Nemytskii operator}, Nonlinear Anal. 30 (1997), 727-734. 
\bibitem{papa2} N. Papageorgiou, {\it On measurable multifunctions with applications to random multivalued equations}, Math. Japonica 32 (1987), 437-464.
\bibitem{papa3} N. Papageorgiou, F. Papalini, {\it Periodic and boundary value problems for second order differential equations}, Proc. Indian Acad. Sci. Math. Sci. 111 (2001), 107-125.
\bibitem{pietkun} R. Pietkun, {\it Structure of the solution set to Volterra integral inclusions and applications}, J. Math. Anal. Appl. 403 (2013), 643-666.
\bibitem{zyg} W. Zygmunt, {\it On superpositionally measurable semi-Carath\'eodory multifunctions}, Comment. Math. Univ. Carolin. 33 (1992), 73-77.
\end{thebibliography}
\end{document}